\documentclass[reqno]{amsart}
\usepackage{amscd,amssymb}
\usepackage[mathcal]{euscript}
\usepackage{graphicx}
\usepackage{hyperref}
\usepackage[arrow, matrix, curve]{xy}
\numberwithin{equation}{section}

\newtheorem{Thm}[equation]{Theorem}
\newtheorem{Lemma}[equation]{Lemma}
\newtheorem{Prop}[equation]{Proposition}

\newtheorem{Cor}[equation]{Corollary}
\theoremstyle{definition}
\newtheorem{Def}[equation]{Definition}
\newtheorem{Rmk}[equation]{Remark}
\newtheorem{ex}[equation]{Example}

\newcommand{\R}{\mathbb{R}}

\newcommand{\Z}{\mathbb{Z}}
\newcommand{\C}{\mathbb{C}}
\newcommand{\A}{\mathbb{A}}

\newcommand{\jet}{\mathbb{J}}

\newcommand{\Hom}{\mathrm{Hom}}
\newcommand{\id}{\mathrm{id}}
\newcommand{\SMan}{\mathsf{SM}}
\newcommand{\Man}{\mathsf{M}}

\newcommand{\Aut}{\mathrm{Aut}}
\newcommand{\uAut}{\underline{\mathrm{Aut}}}
\newcommand{\Spin}{\mathrm{Spin}}

\newcommand{\Hor}{\mathcal{H}}
\newcommand{\cHor}{\mathcal{CH}}
\newcommand{\uHom}{\underline{\mathrm{Hom}}}
\DeclareMathOperator{\im}{\text{im}}

\begin{document}
\author{Dominik Ostermayr}
\address{Institut f\"ur theoretische Physik, Universit\"at zu K\"oln, Z\"ulpicher Str. 77, 50937 K\"oln, Germany}
\email{dosterma@math.uni-koeln.de}
\title{Automorphism supergroups of supermanifolds}

\begin{abstract}
A classical theorem states that the group of automorphisms of a manifold $M$
preserving a $G$-structure of finite type is a Lie group.
We generalize this statement to the category of $cs$ manifolds and give some examples,
some of which being generalizations of classical notions, others
being particular to the super case.
Notably, we have to introduce a new notion of supermanifolds
which we call mixed supermanifolds.
\end{abstract}
\maketitle

\tableofcontents

\section{Introduction}
In this article, we study geometric structures on $cs$ manifolds
and their automorphisms.
Super-Riemannian structures on $cs$ manifolds play a prominent role in the work of
Zirnbauer \cite{Zirnbauer}.
In particular, the so-called Riemannian symmetric superspaces are worth mentioning.
Other instances of geometric structures on supermanifolds appeared in the context of supergravity theories \cite{Lott}.

By a geometric structure on a manifold $M$ we mean a reduction of the structure
group of the frame bundle $L(M)$ to some closed subgroup $G\leqslant GL(V).$
Depending on the context, there might be additional conditions like $1$-flatness.
A classical theorem states that the group of automorphisms of such a $G$-structure
is a Lie group provided it is of finite type. (See \cite{Kobayashi} and the references therein.)
This includes for instance the isometry group of a Riemannian manifold.

In this work, we study the analogous structures in the category of $cs$ manifolds (cf. \cite{Deligne Morgan}).
First, we lay the necessary foundations for the definition of a $G$-structure.
This leads naturally to the notion of mixed supermanifolds as follows.
The frame bundle of an ordinary manifold locally modelled on the vector space $V$ is obtained from a cocycle $U_{ij}\rightarrow GL(V)$ by glueing.
Suppose $M$ is a $cs$ manifold (called supermanifolds in this article)
which is locally modelled on the super vector space $V_{\bar{0}}\oplus V_{\bar{1}}.$
Here, $V_{\bar{0}}$ is a real and $V_{\bar{1}}$ is a complex vector space.
Then the analogous cocycle takes values in the mixed Lie supergroup $GL(V)$ which has as body
the mixed manifold $GL(V_{\bar{0}})\times GL(V_{\bar{1}}).$
It is crucial to keep the complex analytic structure on the second factor.
After having developed the basic theory of mixed supermanifolds,
one can define $G$-structures, prolongations and $G$-structures of finite type
along the lines of the classical definitions.
Our main result concerns the functor of automorphisms of a $G$-structure of finite type that
is in addition admissible.
In this situation, if restricted to purely even supermanifolds, the functor is representable by a
mixed Lie group and it is finite dimensional in the sense that the higher
points are determined by the Lie superalgebra of infinitesimal automorphisms of the $G$-structure, which
we prove to be finite dimensional.
Representability can fail for two reasons here,
due to the fact that the higher points of the functor of automorphisms contain all infinitesimal automorphisms of the $G$-structure.
For a representable functor these are necessarily all complete and decomposable,
which means that they admit a decomposition of the form $X + i Y$ for two real complete vector fields.
The theory of $G$-structures can be developed for real supermanifolds
without need for enlarging the category.
Moreover, there is no need for imposing an additional
property on a $G$-structure of finite type.
The only obstruction for the representability of the functor of automorphisms
of finite type is the completeness of the infinitesimal automorphisms.

The paper is organized as follows.
In Section $2$ we introduce mixed supermanifolds.
After giving the basic definitions, we give a short account on mixed Lie supergroups
and principal bundles.
We then show that mixed supermanifolds are the natural home for constructions
such as tangent bundles and frame bundles as well as their mixed forms,
the real tangent bundles and real frame bundles.
In contrast to what the name suggests, mixed supermanifolds are not supermanifolds with
extra structure as we show in Proposition \ref{prop: non-existence of forgetful functor}.
Moreover, we prove that, for our purposes, mixed supermanifolds cannot be avoided
(Proposition \ref{prop: non-existence of functorial lift from complex functions to two real functions}).

In Section $3$ we define a geometric structure to be a reduction
of the real frame bundle of a mixed supermanifold and construct its prolongation.
In the super context it is advisable to make the constructions in such a way that
functoriality is evident. A subtlety is that the standard prolongation
has to be refined to a real prolongation, which is again a geometric
structure in the sense of our definition.
The existence is ensured if the $G$-structure is admissible.
          
In Section $4$ we define the functor of automorphisms of a $G$-structure.
Due to functoriality, prolongation gives rise to inclusions of functors
of automorphisms.
Then we treat the case of a $\{1\}$-structure.
We show that the underlying functor is representable and
the Lie superalgebra of infinitesimal automorphisms is finite dimensional.
An important ingredient is that even real vector
fields possess a flow as we show in Section \ref{subsec: flows of even real vector fields}.
Similar results on the functor of automorphisms of an admissible $G$-structure of finite type
can then be deduced by embedding it into the functor of automorphisms of a $\{1\}$-structure.

Everything we have said has a direct analogue in the category of real supermanifolds,
except that there are no complications caused by mixed structures
and admissibility. The completeness issues remain.
The analogous theorems are stated in Section $5.$

Finally, in Section $6$ we discuss some examples. We treat even and odd metric structures on supermanifolds
and construct a canonical admissible geometric structure of finite type associated
to the superization of a Riemannian spin manifold as studied in \cite{Cortes et al, Lott}.

\addtocontents{toc}{\protect\setcounter{tocdepth}{1}}
\subsection*{Acknowledgements}
This research was funded by SFB/TR 12 ``Symmetries and Universality in Mesoscopic Systems''.
I would like to thank Alexander Alldridge for helpful comments on earlier drafts of this paper.
\addtocontents{tox}{\protect\setcounter{tocdepth}{2}}

\section{Recollections on supergeometry}
\label{sec: Recollections on supergeometry}
\subsection{Mixed supermanifolds}
A \emph{complex super vector space} 
is a $\Z/2$-graded complex
vector space
$V = V_{\bar{0}}\oplus V_{\bar{1}}.$
A morphism is a grading preserving complex linear homomorphism.
The resulting category is closed symmetric monoidal
with respect to the evident notion of tensor product and inner hom objects.

A \emph{general mixed super vector space} consists of the data $(V, V_\R, V_\C)$
where $V$ is a complex super vector space, $V_\R\subseteq V$ is a real sub super vector space,
$V_\C\subseteq V$ is a complex sub super vector such that
$V_\C\subseteq V_\R$ and the canonical map $\C\otimes V_\R/V_\C \rightarrow V/V_\C$
is an isomorphism.
A \emph{mixed super vector space} is a general mixed super vector space $(V, V_\R, V_\C)$
such that $(V_\R)_{\bar{1}} = (V_\C)_{\bar{1}} = V_{\bar{1}}.$
The class of these contains the classes of \emph{super vector spaces} and \emph{complex super vector spaces}
as the extreme cases where $V_{\C} = V_{\bar{1}}$ and $V_{\C} = V$, respectively.
A \emph{real super vector space}
is a general mixed super vector space of the form $(V, V_\R, 0).$
For our purposes it is not necessary to discuss the various
notions of morphisms of general mixed super vector spaces at this point.

\begin{ex}
 One way to produce (general) mixed super vector spaces is the following.
 Suppose given a real sub super vector space $V_\R$ of a complex super vector space $W.$
 The kernel of the induced map $f\colon \C\otimes V_\R\rightarrow W$ is of the form
 $\tilde{V}_\C = \{i\otimes v - 1\otimes iv\mid v\in V_\C\}\cong \overline{V_\C}$ for a complex subspace
 $V_\C\subseteq W$ contained in $V_\R.$
 Then $(V = \im(f), V_\R, V_\C)$ is a general mixed super vector space.
 Of course, conversely, given a general mixed super vector space $(V, V_\R, V_\C)$,
 $V_\C$ can be recovered from this by applying this procedure to
 $V_\R\rightarrow V.$
 In particular, the pair $(V, V_\R)$ determines $V_\C$
and the pair $(V_\R, V_\C)$ determines $V$
up to isomorphism.
\end{ex}

This leads to various notions of supermanifolds.
We will first introduce the relevant notions at the level
of manifolds (without grading).
Consider a (purely even) mixed vector space $V_\C\subseteq V_\R\subseteq V.$
We denote by $\A(V_\R)$ the locally ringed space over $\C$ given by the topological space
$V_\R$ together with the
sheaf $\mathcal{O}_{V_\R}$ of partially holomorphic functions, i.e.
complex valued smooth functions whose differential is complex linear in the fibre $\A(V_\R)\times V_\C\subseteq \A(V_\R) \times V_\R = T\A(V_\R).$

\begin{Rmk}
More concretely, if we choose an isomorphism $V \cong \C^{n}\times \C^{m}$
such that $V_\R \cong \R^{n}\times \C^{m}$ and
$V_\C  \cong \C^{m}$, then these are complex smooth functions $\psi(x, z)$ on
open subspaces of $\R^{n}\times \C^{m}$
which are holomorphic in $z.$
\end{Rmk}

\begin{Def}
A \emph{mixed manifold} consists
of a locally ringed space $(M_0, \mathcal{O}_{M_0})$ over $\C$
with second countable Hausdorff base which
is locally isomorphic to $\A(V_\R)$ for some mixed vector space $(V, V_\R, V_\C).$
The subsheaf of real-valued functions is denoted by $\mathcal{O}_{M_0, \R}.$
The full subcategory of locally ringed spaces over $\C$ with objects mixed manifolds
is denoted by $\Man^\mu.$
\end{Def}

\begin{Rmk}
These are precisely the smooth manifolds
locally of the form $\R^{n}\times \C^{m}$ with transition functions
$(x, z)\mapsto (\varphi(x), \psi(x, z))$, where
$\psi(x, z)$ is holomorphic in $z.$
Put differently, these are manifolds endowed with a Levi flat CR-structure (cf. \cite{Eastwood}).
\end{Rmk}

Consider now a mixed super vector space $(V, V_\R, V_\C).$
We denote by $\A(V_\R)$ the locally ringed superspace over $\C$ given by
the topological space ${V_{\R}}_{\bar{0}}$ together with the structure sheaf
$\mathcal{O}_{\A({V_{\R}}_{\bar{0}})}\otimes_\C \bigwedge V_{\bar{1}}^*.$
Given a mixed super vector space $(V, V_\R, V_\C)$, we can forget the
mixed structure and consider the mixed super vector space $(V, V, V).$
The associated locally ringed space will be denoted by $\A(V).$

\begin{Def}
A \emph{mixed supermanifold} consists of
a locally ringed superspace $M = (M_0, \mathcal{O}_M)$ over $\C$ with
second countable Hausdorff base
which is locally isomorphic to
$\A(V_\R)$ for some mixed super vector space $(V, V_\R, V_\C).$
The full subcategory of locally ringed superspaces
over $\C$ with objects mixed supermanifolds
is denoted by $\SMan^\mu.$
The category $\SMan^\mu$ contains the full subcategories $\SMan$
and $\SMan^\C$ of \emph{supermanifolds} and
\emph{complex supermanifolds} as the extreme cases where ${V_{\C}} = V_{\bar{1}}$ and
$V_\C = V$, respectively.
\end{Def}

The sheaf of nilpotent functions on a mixed (real) supermanifold $M$
will be denoted by $\mathcal{J}_M.$ The mixed (real) supermanifold
structure on $M$ induces the structure of a mixed (real) manifold on the locally ringed space
$(M_0, \mathcal{O}_{M}/\mathcal{J}_M)$
which we abbreviate by abuse of notation by $M_0.$
Moreover, we set $\mathcal{O}_{M_0} := \mathcal{O}_M/\mathcal{J}_M.$
Then the inclusion $i\colon \Man^\mu \rightarrow \SMan^\mu$
has the right adjoint $r\colon \SMan^\mu\rightarrow \Man^\mu$, $M\mapsto M_0.$

Given a mixed supermanifold, we define the sheaf of \emph{real functions}
to be the pullback in the square of (real) supercommutative superalgebras
\[
\xymatrix{
 \mathcal{O}_{M, \R}  \ar[r]\ar[d] & \mathcal{O}_{M_0, \R} \ar[d]\\
 \mathcal{O}_M \ar[r] & \mathcal{O}_{M_0}.
}
\]

We will often consider a mixed
supermanifold as a set-valued functor
on $\SMan^\mu$
by the assignment $T\mapsto \SMan^\mu(T, M).$
Then there is a natural transformation of functors $M\rightarrow r^*i^* M = r^*M_0$
which is given by sending a map $T\rightarrow M$ to its underlying map
$T_0\rightarrow M_0.$ The second part of the next lemma is only the first encounter
of the typical reality condition enforced by a mixed structure.

\begin{Lemma}
\label{lem: reality condition}
Consider a mixed super vector space $(V, V_\R, V_\C).$
\begin{itemize}

 \item[(a)] There is a natural isomorphism $\SMan^\mu(M, \A(V)) = \Gamma(\mathcal{O}_M\otimes_\C V_\C)_{\bar{0}}.$

 \item[(b)] The following diagram is a pullback of functors on $\SMan^{\mu}$:
\[
 \xymatrix{
\A(V_\R)\ar[r]\ar[d]\ar[d] & r^*\A({V_{\R}}_{\bar{0}})\ar[d]\\
\A(V)\ar[r] & r^*\A(V_{\bar{0}}).
}
\]

\item[(c)] In other words, we have
\[
\SMan^\mu(M, \A(V_\R)) \cong \Gamma(\mathcal{O}_{M, \R, \bar{0}}\otimes_\R ({V_{\R}/V_\C})_{\bar{0}}) \oplus \Gamma(\mathcal{O}_{M, \bar{0}}\otimes_\C (V_\C)_{\bar{0}})\oplus \Gamma(\mathcal{O}_{M, \bar{1}}\otimes_\C V_{\bar{1}}).
\]
\end{itemize}
\end{Lemma}
\begin{proof}
The proof is similar as in \cite[Theorem 4.1.11]{CCF}.
\end{proof}
\begin{Cor}
The category $\SMan^\mu$ admits all finite products
and the full subcategory $\Man^\mu$ is closed under finite products in $\SMan^\mu.$
\end{Cor}

Let $M_0$ be a mixed manifold.
Consider the sheaf $\mathcal{T}_{M_0}$ whose sections
over $U_0$ are complex
linear derivations of $\mathcal{O}_{M_0}|_{U_0}$
and the subsheaf $\mathcal{T}_{M_0, \R}$
of those derivations which restrict to derivations of $\mathcal{O}_{M_0, \R}|_{U_0}.$
Then $\mathcal{T}_{M_0, \R}$ contains
a complex ideal $\mathcal{T}_{M_0, \C}$ of derivations which annihilate
$\mathcal{O}_{M_0, \R}|_{U_0}.$ The quotient by this sheaf is (non-canonically)
isomorphic to the sheaf of derivations of $\mathcal{O}_{M_0, \R}.$

Now, if $M$ is a mixed supermanifold, the complex tangent sheaf is the sheaf $\mathcal{T}_M$
whose sections over $U_0$ are the complex linear superderivations of $\mathcal{O}_M|_{U_0}.$
In analogy with the definition of the real functions, one defines the \emph{real tangent sheaf}
by the pullback
\[
 \xymatrix{
\mathcal{T}_{M, \R} \ar[r]\ar[d] & \mathcal{T}_{M_0, \R}\ar[d]\\
\mathcal{T}_M \ar[r] & \mathcal{T}_{M_0},
}
\]
where the lower arrow takes a vector field to its underlying vector field.

An important point is that, although $\mathcal{T}_{M, \R}$ is not closed
under brackets, its even part is and consists of those derivations which restrict
to derivations of $\mathcal{O}_{M, \R}.$
In analogy, one defines $\mathcal{T}_{M, \C}\subseteq \mathcal{T}_{M ,\R}$
in terms of $\mathcal{T}_M$, $\mathcal{T}_{M_0}$ and $\mathcal{T}_{M_0, \C}.$
Then $(\mathcal{T}_{M, \C})_{\bar{0}}\subseteq (\mathcal{T}_{M, \R})_{\bar{0}}$ is an ideal.

The tangent space $T_m M$ at $m\in M_0$ is the complex super vector space of complex derivations
$\mathcal{O}_{M, m}\rightarrow \C.$
This comes with a mixed structure by considering
the real subspace $(T_m M)_\R$
consisting of those derivations
which induce a derivation $\mathcal{O}_{M_0, \R, m}\rightarrow \R$
together with its complex subspace $(T_m M)_\C$ of those derivations in $(T_mM)_\R$
which vanish on $\mathcal{O}_{M_0, \R, m}.$

\subsection{Mixed Lie supergroups}
In this section we give a brief review on basic results concerning
mixed Lie supergroups.

\subsubsection{Equivalence of mixed Lie supergroups and mixed super pairs}

\begin{Def}
A  \emph{mixed Lie supergroup} is a group object in $\SMan^\mu.$
\end{Def}

First we characterize mixed Lie groups, i.e.~mixed Lie supergroups with trivial
odd direction.
For a real (resp.~mixed) Lie group $G$ we will use the notation $\mathrm{Lie}_\R(G)$
(resp.~$\mathrm{Lie}_\C(G)$) for the Lie algebra
of left-invariant derivations of the sheaf of real valued smooth functions
(resp.~sheaf of complex functions).

We define a \emph{mixed pair} to be a pair
$(\mathfrak{g}_\C, G^{sm})$
consisting of a real Lie group $G^{sm}$
and an $\mathrm{Ad}_{G^{sm}}$-invariant ideal
$\mathfrak{g}_\C\subseteq \mathrm{Lie}_\R(G^{sm})$
endowed with a complex structure which is respected by the adjoint
action of $G^{sm}.$ 

A morphism of such pairs is a morphism of Lie groups such that
the differential at the identity
respects the complex ideals.

\begin{Lemma}
 The categories of mixed Lie groups and mixed pairs are equivalent.
\end{Lemma}
\begin{proof}
This follows from the Baker--Campbell--Hausdorff formula
as in the case of complex analytic structures on Lie groups.
\end{proof}

As usual, the adjoint representation of a mixed Lie group $G$ is the differential at the identity
of the conjugation action of $G$ on itself.
It can be seen as a mixed morphism
$G\times \A(\mathfrak{g}_\R)\rightarrow \A(\mathfrak{g}_\R).$

Now, we turn our attention to mixed Lie supergroups.
A \emph{mixed super pair} consists of a pair
$(\mathfrak{g}, G_0)$ where $G_0$ is a mixed Lie group
and $\mathfrak{g}$ is complex Lie superalgebra together with
\begin{itemize}
 \item[(a)] an isomorphism $\mathrm{Lie}_\C(G_0) \cong \mathfrak{g}_{\bar{0}}$, and
 \item[(b)] an action $\sigma\colon G_0\times \A(\mathfrak{g})\rightarrow \A(\mathfrak{g})$ such that
            $\sigma(g)|_{\A(\mathfrak{g}_{\bar{0}, \R})} = \mathrm{Ad}_{G_0}$ and the differential
            of $\sigma$ acts as the adjoint representation
\[
d\sigma(X)(Y) = [X, Y]. 
\]
\end{itemize}

There is an evident notion of a morphism of mixed super pairs
and the following result follows along the same lines as the corresponding
for real and complex Lie supergroups.
\begin{Prop}
\label{prop: equivalence pairs and groups}
 The categories of mixed super pairs and mixed Lie supergroups are equivalent.
\end{Prop}
\begin{proof}
 See \cite[7.4]{CCF}.
\end{proof}

An important notion is the following.

\begin{Def}
A \emph{mixed real form} of a complex Lie supergroup $G$
is a mixed Lie supergroup $G_\R$
together with a group morphism $i\colon G_\R \rightarrow G$
such that $i_0\colon (G_\R)_0\rightarrow G_0$
is the inclusion of a closed subgroup
and $di_e\colon T_e(G_\R)\rightarrow T_e(G)$
is an isomorphism.
\end{Def}
\begin{Rmk}
 Any mixed real form $G_\R\leqslant G$ yields a 
 mixed real form $(G_\R)_0\leqslant G_0.$
 Conversely, given a mixed real form $(G_0)_\R\leqslant G_0$,
 the pullback
 \[
 \xymatrix{
  G_\R \ar[r] \ar[d] & r^* (G_0)_\R\ar[d]\\
  G\ar[r] & r^*G_0
}
  \]
 is representable and defines a mixed real form of $G.$
For that reason, we will adopt the notation $(G_\R)_0 = (G_0)_\R = G_{0, \R}.$
 \end{Rmk}

\begin{ex}
Finally, we come to discuss the example of linear supergroups.
Let $(V, V_\R, V_\C)$ be a mixed super vector space.
Then we have the complex Lie supergroup
$GL(V)$ given by the complex group $GL(V_{\bar{0}})\times GL(V_{\bar{1}})$
and the Lie superalgebra $\mathfrak{gl}(V).$
An element of
$GL(V)(T)$ is given by an automorphism over $T$ of the trivial vector bundle $\underline{V}_T = T\times \A(V)\rightarrow T.$

Consider the subgroups
of those even invertible isomorphisms
of $V$ respecting $V_\C$ or the pair $V_\C \subseteq V_\R.$
We will denote them by
\[
GL^\mu(V)_{0,\R} \leqslant GL^\mu(V)_{0}\leqslant GL(V)_0.
\]

We then define the two group-valued functors
${GL^\mu(V)}$ and $GL^\mu(V)_{\R}$ on $\SMan^\mu$ by the pullback
\[
\xymatrix{ 
GL^\mu(V)_{(\R)}\ar[r]\ar[d] & r^*{GL^\mu(V)_{0, (\R)}}\ar[d]\\
 GL(V)\ar[r] & r^*{GL(V)_0},
}
\]
where it is understood that the quantities in parentheses are only present in the latter case.

The inclusion
$\mathrm{Lie}_\R(GL^\mu(V)_{0, \R})\subseteq \mathfrak{gl}(V)_{\bar{0}}$
only defines a mixed structure
in the cases 
$V_\C = V_{\bar{1}}$ and
$V_\C = V.$
In this case
$GL^\mu(V)_{\R}$
is representable
and is a mixed real form of $GL(V).$
In general, $GL^{\mu}(V)_{(\R)}$ is not representable.
\end{ex}

\subsubsection{Actions of mixed Lie supergroups and their point functors}
A \emph{left action} of the mixed Lie supergroup $G$ on the mixed supermanifold $M$ is given by a unital
and associative map
$a\colon G\times M\rightarrow M.$
The map $a^\sharp$
can be made explicit in terms of two more basic objects.
First, let $\underline{a}$ denote the action $G_0\times M\rightarrow G\times M\rightarrow M.$
Then any $g\in G_0$ (considered as a map $g\colon\A(\{0\})\rightarrow G$) gives a map
\[
\xymatrix{
a_g\colon M \cong \A(\{0\}) \times M \ar[r]^-{g\times M} & G_0\times M\ar[r]^-{\underline{a}} & M.
}
\]
Secondly, the action gives rise to a Lie superalgebra antimorphism
\begin{gather}
\label{def rho}
\xymatrix{
\rho\colon\mathfrak{g}\ar[r] & \Gamma(\mathcal{T}_M),\ X\mapsto (e\times M)^\sharp \circ (X\otimes 1)\circ a^\sharp
}
\end{gather}
and we have
\begin{itemize}
 \item[(a)] $\rho|_{\mathfrak{g}_{\bar{0}}}(X) = (X\otimes 1)\circ \underline{a}^\sharp$, and
 \item[(b)] $\rho(g\cdot Y) = (a_g^{-1})^\sharp\cdot \rho(Y)\cdot a_g^\sharp.$
\end{itemize}
Conversely, given an action $\underline{a}\colon G_0\times M\rightarrow M$
and $\rho$ satisfying $(a)$ and $(b)$, then one can construct an action $G\times M\rightarrow M$
(cf. \cite[Propositions 8.3.3, 8.3.2]{CCF}).

Now let $G$ be a mixed Lie group
and $M$ a mixed supermanifold
and consider an action $a^{sm}\colon G^{sm}\times M\rightarrow M.$
This gives rise to a Lie algebra morphism
$\mathfrak{g}_\R\rightarrow \Gamma({\mathcal{T}_{M, \R}})_{\bar{0}}.$
The connection between such an action and an action of $G$
is made precise in the next lemma.

\begin{Lemma}
\label{lem: action underlying real Lie group and mixed group}
 The action $a^{sm}$ extends to an action $a\colon G\times M\rightarrow M$
if and only if $\mathfrak{g}$ fits into the following square
\[
\xymatrix{
 \mathfrak{g}_\R \ar[r]\ar[d] & \Gamma({\mathcal{T}_{M, \R}})_{\bar{0}}\ar[d]\\
 \mathfrak{g} \ar[r] &          \Gamma(\mathcal{T}_M)_{\bar{0}}.
}
\]
the lower horizontal arrow being an antimorphism of complex Lie algebras.
The extension is unique if it exists.
Equivalently, the restriction of the upper horizontal arrow to $\mathfrak{g}_\C$
factors as a complex linear map through $\Gamma({\mathcal{T}_{M, \C}})_{\bar{0}}.$
\end{Lemma}
\begin{proof}
Uniqueness is clear since any element $X\in \mathfrak{g}$ can be written
in the form $X_1 + i X_2$
for some $X_j\in \mathfrak{g}_\R.$
If the extension in the diagram exists, then the differential $TG^{sm}_\R\times TM\rightarrow TM$
is complex linear on $TG^{sm}_\C\times TM\rightarrow TM$,
which proves that the action extends to $G\times M.$
\end{proof}

Let $T$ be an arbitrary mixed 
supermanifold.
Consider a morphism $\varphi_0\colon T\rightarrow G$ and a homogeneous
derivation
$X\colon \mathcal{O}_G\rightarrow ({e_T}_0)_*\mathcal{O}_T$ along
$e_{T}\colon T\rightarrow *\rightarrow G.$

Given this, we construct a homogeneous derivation along $\varphi_0$ as follows:
\[
\xymatrixcolsep{6pc}
\xymatrix{
\varphi_0\cdot X\colon \mathcal{O}_G  \ar[r]^-{(\varphi_0\times T)^\sharp \circ (1\otimes X)\circ \mu^\sharp} & (\mu_0)_* {(\varphi_0\times e_T)_0}_*\mathcal{O}_{T\times T}\ar[r]^-{\Delta^\sharp} & ({\varphi}_0)_*\mathcal{O}_T.
}
\]
Similarly, for two homogeneous derivations $X$ and $Y$ we set
\[
X\cdot Y := \Delta^\sharp \circ (\mu_0)_*((X\otimes 1)\circ (1\otimes Y)) \circ \mu^\sharp.
\]
Now, suppose $G$ acts on $M$ and let $X$ and $Y$ be as above.
We set
\[
\xymatrixcolsep{5pc}
\xymatrix{
 \rho(X)\colon \mathcal{O}_{T\times M}  \ar[r]^-{(1\otimes X\otimes 1)\circ (T\times a)^\sharp} & \mathcal{O}_{T\times T\times M}\ar[r]^-{(\Delta\times M)^\sharp} & \mathcal{O}_{T\times M}.
}
\]

Then $\rho(X)$ is the $\mathcal{O}_T$-linearization of
$\rho(X)\circ p_T^\sharp$, where $p_T\colon T\times M\rightarrow M$
is the projection.
From the associativity of the action it follows that
\[
\rho(X\cdot Y) = (-1)^{|X||Y|}\rho(Y)\circ \rho(X).
\]

Let $n\geq 0$, then $\Gamma(\mathcal{O}_{\A(\C^{0|n})})$
is the exterior algebra on generators $\eta_i.$
As usual, given a non-empty subset
$I\subset \{1, \ldots, n\}$,
we set $\eta^I  = \prod_{i\in I} \eta_i$,
where we implicitly use the ordering on $I$ induced from the standard ordering
on $\{1, \ldots, n\}.$
\begin{Lemma}
\label{lem: family of automorphisms from group action}
Suppose $G$ is mixed and acts on the mixed supermanifold $M.$
\begin{itemize}
\item[(a)]
Any $\varphi\in G(\A(\C^{0|n})\times T)$
is uniquely determined by $\varphi_0\in G(T)$ and homogeneous derivations $X_I$ along $e_{T}$ of degree $|I|$
and
\[
\varphi^\sharp = \varphi_0^\sharp \cdot \prod_{k = 1}^n\Big( 1 + \sum_{k\in I\subseteq\{1, \ldots, k\}}\eta^I X_I\Big).
\]
\item[(b)]
Moreover, under this identification, the morphism $a_{\varphi}$, defined as the composition
\[
\xymatrix{
( \A(\C^{0|n})\times T \times  a)\circ ((\mathrm{pr}_{\A(\C^{0|n})\times T}, \varphi)\times M)\colon \A(\C^{0|n})\times T\times M\ar[r] & \A(\C^{0|n})\times T\times M,
}
\]
takes the form
\[
a_{\varphi}^\sharp = \prod_{k = n}^1 \Big(1 + \sum_{k\in I \subseteq\{1, \ldots, k\}} \eta^I \rho(X_I)\Big)\cdot a_{\varphi_0}^\sharp.
\]
\end{itemize}
\end{Lemma}
\begin{proof}
 The first part is proved by induction on $n$ and the second part then boils down to
 $(\mu\times M)^\sharp \circ a^\sharp = (G\times a)^\sharp \circ a^\sharp.$
\end{proof}

\subsection{Mixed real forms of principal $G$-bundles}

Suppose given a mixed supermanifold $M$ and a group-valued functor $G$
on $\SMan^\mu.$
A principal $G$-bundle is a functor $P$ on $\SMan^\mu$ together
with a right $G$-action and a map $\pi\colon P\rightarrow M$
equivariant with respect to the trivial action on $M$
such that for each $m\in M_0$ there exist an open neighbourhood
$U$ and equivariant isomorphisms $U\times G\rightarrow P|_U$
over $U.$ This reduces to the usual definition if $G$ is representable.

Later we will need to build real forms of certain principal bundles.
This will be done so with the help of the following lemma.

\begin{Lemma}
\label{lem: real forms of principal bundles}
Let $G$ be a complex Lie supergroup with mixed real form $G_\R.$
Let $P\rightarrow M$ be a principal $G$-bundle over a mixed supermanifold $M$ and $P_{0, \R}\rightarrow P_{0}$
a reduction of $P_0$ to $G_{0, \R}.$
Then the pullback
\[
\xymatrix{
 P_\R \ar[d]\ar[r] & r^*(P_{0, \R})\ar[d]\\
 P \ar[r] & r^*(P_0).
}
\]
is a principal $G_\R$-bundle.
\end{Lemma}
\begin{proof}
We observe that $G_\R$ acts on $P_\R$ by the universal property
of the pullback and the map $P_\R\rightarrow P\rightarrow M$ is equivariant
with respect to this action.
So we only need to show local triviality.
We choose trivializations $\psi_i\colon U_i \times G\rightarrow P|_{U_i}$ on coordinate charts $U_i = \A(V_\R)$
on $M.$ They come with retractions $r_i\colon U_i\rightarrow (U_i)_0.$
Without loss of generality, we may assume that ${P_{0, \R}}|_{(U_i)_0}$ is trivial, too,
say by maps $\varphi_i\colon (U_i)_0\times G_{0 \R} \rightarrow {P_{0, \R}}|_{(U_i)_0}.$
The $\varphi_i$ induce trivializations $\tilde{\varphi}_i\colon (U_i)_0 \times G_0\rightarrow {P_{0}}|_{(U_i)_0}$
which differ from $(\psi_i)_0$ by maps $g_i\colon (U_i)_0\rightarrow G_0$
in the sense that
\[
\xymatrix{
 \tilde{\varphi}_i = (\psi_i)_0 \circ ((U_i)_0\times a_0)\circ  ((\id_{(U_i)_0}, g_i)\times G_0)\colon (U_i)_0\times G_0\ar[r] & {P_0}|_{(U_i)_0}.
}
\]
Denoting by $\tilde{a}$ the composition $G_0\times G\rightarrow G\times G\rightarrow G$,
we now set
\[
 \xymatrix{
\tilde{\psi}_i = \psi_i \circ (U_i \times \tilde{a}) \circ (U_i \times g_i \times G)\circ((\id_{U_i}, r_i)\times G)\colon U_i\times G\rightarrow P|_{U_i},
}
\]
which is still a trivialization.
Then $(\tilde{\psi}_i)_0 = \tilde{\varphi}_i$,
and the universal property of the pullback now shows that
$\tilde{\psi}_i$, restricted to $U_i\times G_\R$, gives a trivialization of ${P_\R}|_{U_i}.$
\end{proof}

\subsection{Tangent bundles and frame bundles of mixed supermanifolds}

Suppose $M$ is a mixed supermanifold
locally modelled on the mixed super vector space $(V, V_\R, V_\C).$
The sheaf $\mathcal{T}_M$ is locally free on $V$
and glueing leads to the mixed 
total space $TM\rightarrow M.$
If $i\colon M_0 \rightarrow M$ is the canonical inclusion,
then
\[
i^* TM = TM_{\bar{0}}\oplus TM_{\bar{1}}
\]
for certain complex bundles $(TM)_{\bar{j}}\rightarrow M_0$ (in
the category of mixed manifolds).
Actually, we have $(TM)_{\bar{0}} = TM_0.$

Define $\underline{V}_T = T\times \A(V)\rightarrow T$ to be the trivial
vector bundle over $T$ with fibre $\A(V).$
There is a vector bundle of homomorphisms $\Hom(\underline{V}_M, TM)\rightarrow M$
and the $T$-points of the total space
are given by squares of vector bundles
\begin{gather}
\label{functor of points 1 jets}
 \xymatrix{
\underline{V}_T\ar[r]^{\varphi'}\ar[d] & TM\ar[d]\\
T \ar[r]^f & M.
}
\end{gather}
Equivalently, a $T$-point consists of a tuple $(f, \varphi)$ consisting of a map $f\colon T\rightarrow M$
and a map $\varphi\colon \underline{V}_T\rightarrow f^*(TM)$ of vector bundles over $T.$

Another point of view is the following.
Let $\jet_V^k$ be the $(k+1)$-truncation of the free supersymmetric algebra on $V$:
\begin{gather*}
\jet^k_V := \mathrm{Sym}(V)/ \langle V\rangle^{k+1}.
\end{gather*}
Define $\mathrm{Spec}(\jet^k_V)$ to be the complex superspace with reduced space a point and
$\jet^k_V$ as algebra of functions.
We set $\jet^k_V M := \underline{\mathrm{Hom}}(\mathrm{Spec}(\jet^k_V), M)$,
i.e.
\[
\SMan^\mu(T, \jet^k_V M) = \mathsf{LRS}_\C(T\times \mathrm{Spec}(\jet^k_V), M),
\]
where $\mathsf{LRS}_\C$ denotes the category of locally ringed superspaces over $\C$,
compare \cite[5.4]{AHW}.
The inclusion $0\rightarrow V$ induces a map $\mathrm{Spec}(\jet^k_V)\rightarrow *$
which in turn induces $\iota_M\colon \jet^k_V M\rightarrow M$
by precomposition.
By inspection we have $\jet^1_VM = \Hom(\underline{V}_M, TM)$ over $M.$

The frame bundle of $M$ is the open subsupermanifold of $\Hom(\underline{V}_M, TM)$
characterized by
\[
 L(M)(T) = \{(f, \varphi) \in \Hom(\underline{V}_M, TM)(T)\mid \varphi\ \text{isomorphism}\}. 
\]
In terms of squares: $(f, \varphi)\in L(M)(T)$ if and only if the associated square (\ref{functor of points 1 jets})
is a pullback.
This is a principal $GL(V)$-bundle over $M.$

We have $L(M)_{0} = L(TM_{\bar{0}})\times_M L(TM_{\bar{1}})$,
and thus the mixed structure of $M$
yields
subbundles
\[
{L^\mu(M)_{0, \R}}\rightarrow L^\mu(M)_0 \rightarrow L(M)_0,
\]
where 
$L^\mu(M)_0$ (resp.~$L^\mu(M)_{0, \R}$)
is the subbundle of those frames which map ${V_\C}$ to $TM_\C$
(resp.~$(V_\R, V_\R)$ to $(TM_\R, TM_\C)$).
By pulling back, we obtain the bundles
\[
 \xymatrix{
L^\mu(M)_{(\R)} \ar[r]\ar[d]& r^* {L^\mu(M)_{0, (\R)}} \ar[d]\\ 
L(M) \ar[r] & r^* L(M)_0.
}
\]

The structure group of
${L^\mu(M)}_{(\R)}$ is precisely $GL^\mu(V)_{(\R)}$,
and this functor of frames
is representable precisely
for supermanifolds and complex supermanifolds,
that is, in terms of local models
$V_\C \in \{V_{\bar{1}}, V\}.$

All these  principal bundles
have associated bundles
that fit in a square
\[
\xymatrix{
 TM_\R\ar[r]\ar[d] & r^*T(M_0)_\R\ar[d]\\
 TM \ar[r] & r^*T(M_0),
}
\]
which is a pullback in
view of the pullback square defining $\mathcal{T}_{M, \R}$ in
terms of $\mathcal{T}_M$, $\mathcal{T}_{M_0}$ and $\mathcal{T}_{M_0, \R}.$

\section{Geometric structures on mixed supermanifolds}
\label{sec: Geometric structures on mixed supermanifolds}
We can now define the notion of a geometric structure on a mixed supermanifold.
Let $G\leqslant GL(V)$ be a closed mixed Lie subgroup,
i.e. $G_0^{sm}\leqslant GL(V)_0^{sm}$ is closed
and $G_0\leqslant GL(V)_0$ is a mixed embedding.

\subsection{Basic definitions}
\begin{Def}
A \emph{$G$-structure} on $M$ is a reduction $P$ of $L^\mu(M)_\R$
to $G.$
Equivalently, it is a reduction $P$ of $L(M)$ such that
$P_0\rightarrow L(M)_0$ factors through $L^\mu(M)_{0, \R}.$
\end{Def}

Any $G$-structure $P$ comes with a canonical $1$-form $\vartheta\colon TP\rightarrow \underline{V}_P.$
It sends a pair $(f, X) \in TP(T)$, considered as the data of a map $f = (\pi \circ f, \varphi)\colon T\rightarrow P$ and
a section $X$ of $f^*(TP)$, to the composite
\[
 \xymatrix{
T\ar[r]^-{X} & f^*(TP)\ar[r]^-{f^*(d\pi)} & (\pi\circ f)^*(TM)\ar[r]^-{{\varphi}^{-1}} & \underline{V}_T \ar[r]^{f\times \id_V} & \underline{V}_P.
}
\]
The differential of the canonical $1$-form $\vartheta\colon TP\rightarrow \underline{V}_P$ is a $2$-form
$d\vartheta\colon \Lambda^2 TP \rightarrow \underline{V}_P.$

\begin{Lemma}
\label{lem: infinitesimal invariace}
Let $\mathcal{V}\colon P\times \mathfrak{g}\rightarrow TP$ be the restriction of the differential
of the action $P\times G\rightarrow P.$
For all $A\colon S\rightarrow \underline{\mathfrak{g}}_P$
and $x\colon S\rightarrow TP$ with same underlying map $S\rightarrow P$
we have
\[
\xymatrix{
d\vartheta \circ (\mathcal{V}(A)\wedge x) = - A(\vartheta(x))\colon S\ar[r] & \underline{V}_P.
}\]
\end{Lemma}
\begin{proof}
This is Proposition $4$ in \cite{Fujio}.
\end{proof}

\begin{Rmk}
Although we will make no use of it, we remark that, in analogy with the usual definition,
one can define a $G$-structure to be \emph{flat} if $M$ can be covered by coordinate charts
$U_i\cong V$ such that the square determined by the coordinates
\[
\xymatrix{
\underline{V}_{U_i}\ar[r]\ar[d] & TM\ar[d]\\
U_i\ar[r]^{\subseteq} & M,
}
\]
which is contained in $L(M)(U_i)$, lies in $P(U_i).$
\end{Rmk}

\subsection{Prolongation}

\subsubsection{Unrestricted prolongation}
Adapting the classical construction \cite{Kobayashi}, we
will in this subsection associate with a $G$-structure $P$ on $M$
a tower of prolongations
\[
 \xymatrix{
\cdots \ar[r] &P^{(k)}\ar[r] & P^{(k-1)} \ar[r] & \cdots \ar[r] & P^{(1)}\ar[r] & P^{(0)} = P \ar[r] & M,
}
\]
where $P^{(i+1)}\rightarrow P^{(i)}$ is a reduction of $L(P^{(i)})$ to $G^{(i+1)}.$
Here
$G^{(0)} = G$ and
$G^{(i)}$ is a vector group for all $i\geq 1.$
\begin{Rmk}
Given a super vector space, the associated supergroup structure on $\A(V)$
will be denoted by $V.$
More generally, if a Lie supergroup $G$ acts linearly on a complex super vector space $V$,
then the associated semi-direct product will be denoted by $G\ltimes V$ instead of $G\ltimes \A(V).$
\end{Rmk}

It will be convenient to introduce a name for the representation of $G$ on $V$: $\alpha\colon G\rightarrow GL(V).$
Applying $\jet^1_V(-)$ to $G\rightarrow P\rightarrow M$ yields a principal $\jet^1_VG$-bundle
$\jet^1_VG\rightarrow \jet^1_VP \rightarrow \jet^1_V M$ and the usual identification
$TG \cong G\ltimes_{\mathrm{ad}} \mathfrak{g}$
gives an isomorphism of groups $\jet^1_V(G)\cong G\ltimes_{\mathrm{ad}} \uHom(V, \mathfrak{g})$,
where $G$ acts via its adjoint representation on $\mathfrak{g}.$
The bundle of horizontal frames is defined by the pullback
\[
\xymatrix{
\Hor \ar[r]\ar[d]^{d\pi_*} & \jet_V^1P\ar[d]\\
P \ar[r] & \jet_V^1M.
}
\]
Its $S$-points are the squares
\[
\xymatrix{
\underline{V}_S\ar[r]^{h}\ar[d] & TP\ar[d]\\
S \ar[r]^f & P
}
\]
such that the composite square 
\[
\xymatrix{
\underline{V}_S\ar[r]\ar[d] & TP\ar[d]\ar[r] & TM\ar[d]\\
S \ar[r] & P \ar[r] & M
}
\]
lies in $P(S).$
Moreover, $\Hor$ is the total space
of a principal $G\ltimes_{\mathrm{ad}} \underline{\mathrm{Hom}}(V, \mathfrak{g})$-bundle
with respect to the map $d\pi_*.$
We need to construct an action of $G\ltimes_{\mathrm{\alpha}} \underline{\mathrm{Hom}}(V, \mathfrak{g}).$
The group $G$ acts via $\alpha$ on $\jet^1_V(P)$ by precomposition.
Together with the action of $\underline{\mathrm{Hom}}(V, \mathfrak{g})\leqslant \jet_V^1(G)$,
this yields an action of $G\ltimes_\alpha\underline{\mathrm{Hom}}(V, \mathfrak{g})$ on $\jet_V^1(P)$,
which restricts to an action on $\Hor.$
The composition $\iota_P\colon\Hor \rightarrow \jet_V^1P\rightarrow P$
is equivariant if we let act $G\ltimes_\alpha\underline{\mathrm{Hom}}(V, \mathfrak{g})$
trivially on $P.$
Moreover, $d\pi_*$ is equivariant with respect to this action if we let act $G\ltimes_\alpha\underline{\mathrm{Hom}}(V, \mathfrak{g})$
on $P$ via the projection to $G.$

The canonical vertical distribution
$\mathcal{V}\colon \underline{\mathfrak{g}}_P\rightarrow TP$ gives rise to a map
$\jet_V^1 P\rightarrow \jet_{V\oplus\mathfrak{g}}^1 P$
and the composition $\Hor\rightarrow \jet_{V\oplus\mathfrak{g}}^1 P$
factors through $L(P).$
Moreover, the $GL(V\oplus \mathfrak{g})$-action
on $L(P)$ is seen to restrict to the action of 
$G\ltimes_\alpha  \underline{\mathrm{Hom}}(V, \mathfrak{g})\leqslant GL(V\oplus \mathfrak{g}).$
This identifies $\iota_P\colon \Hor\rightarrow P$ as a
reduction of $L(P)$ to $G\ltimes_\alpha \uHom(V, \mathfrak{g}).$

As usual, $\mathfrak{g}^{(1)}$ is defined to be the kernel of
the super-antisymmetrizer $\partial\colon \underline{\Hom}(V, \mathfrak{g})\rightarrow \underline{\Hom}(V\otimes V, V)\rightarrow \underline{\Hom}(\Lambda^2 V, V)$,
\[
 (\partial S)(v, w)  := \frac{1}{2}( S(v)(w) - (-1)^{|v||w|} S(w)(v)),
\]
and the first prolongation $P^{(1)}\rightarrow P$ is obtained from $\Hor\rightarrow P$ by two successive reductions
of the structure group to $\mathfrak{g}^{(1)}$ using the following lemma.
\begin{Lemma}
\label{lem: reduction}
Consider a short exact sequence of mixed Lie supergroups
\[
\xymatrix{
1 \ar[r] & H \ar[r] & G \ar[r] & K \ar[r] & 1.  
}
\]
Let $\pi\colon P\rightarrow B$ be a $G$-principal bundle and assume that there
is a $G$-equivariant map $f\colon P\rightarrow K.$
Then $P/H \rightarrow B$ is a principal $K$-bundle and as such isomorphic
to the trivial bundle. Moreover, the map $(\pi, f)\colon P\rightarrow B\times K$
is a principal $H$-bundle.
\end{Lemma}
\begin{proof}
Since any map of principal bundles is an isomorphism,
it suffices to construct a $K$-equivariant map $P/H\rightarrow B\times K$ over $B.$
But such a map can be constructed from the $G$-equivariant map $(\pi, f)\colon P\rightarrow B\times K$
since $H$ acts trivially on the target.
\end{proof}

The first step is a reduction to $\uHom(V, \mathfrak{g})\leqslant G\ltimes_\alpha \uHom(V, \mathfrak{g}).$
We have two maps $d\pi_*$, $\iota_P\colon\Hor\rightarrow P$
over the same map to the base $M.$
Fibrewise comparison yields a map $d\colon \Hor\rightarrow G.$
It follows now from the equivariance properties of $d\pi_*$ and $\iota_P$ that $d$
is $G\ltimes_\alpha\uHom(V, \mathfrak{g})$-equivariant
if we let this group act from the right on $G$ by $g\cdot (g', \varphi) = (g')^{-1}g.$
Now we can apply Lemma \ref{lem: reduction}
and see that $(\iota_P, d)\colon \Hor\rightarrow P\times G$
is a principal $\underline{\mathrm{Hom}}(V, \mathfrak{g})$-bundle.
Pulling back along the inclusion $P\times \{1\}\hookrightarrow P\times G$
yields the bundle of \emph{compatible horizontal frames} $\cHor\rightarrow P$,
a reduction of $L(P)$ to the group $\underline{\mathrm{Hom}}(V, \mathfrak{g}).$
Its $S$-points consist of those squares $(f, h)$ such that $T(\pi)\circ h = f\in P(S).$

The second reduction is a little bit more elaborate.
For a section $v\colon T\rightarrow \underline{V}_T$ and a map $f\colon T\rightarrow P$, we will
use the shorthand
$v_f:=  (f\times \A(V))\circ v\colon T\rightarrow \underline{V}_P.$
\begin{Lemma}
\label{lem: dtheta inner product}
For all compatible horizontal frames $(f, h)\in \cHor(T)$ and all sections
$x\colon T\rightarrow \underline{V}_T$, we have $\vartheta(h(x)) = x_f$:
\[
 \xymatrix{
 T \ar[r]^-{x}\ar[rrd]_{x_f} & \underline{V}_T \ar[r]^{h} & TP\ar[d]^\vartheta\\ 
 & & \underline{V}_P.
}
\]
\end{Lemma}
\begin{proof}
This follows immediately from the definition.
\end{proof}

Consider $(f, h)\in \cHor(S).$ The \emph{torsion} is defined to be
the composition
\[
\entrymodifiers={+!!<0pt,\fontdimen22\textfont2>}
\xymatrix{
 c(f, h)\colon \Lambda^2 \underline{V}_S\ar[r]^-{\Lambda^2h} & \Lambda^2 TP \ar[r]^{d\vartheta} & \underline{V}_P.
}
\]
Equivalently, it is given by a map
\[
\xymatrix{
c'(f, h)\colon S\ar[r] & \underline{\mathrm{Hom}}(\Lambda^2 V, V) .
}
\]
By naturality, we obtain a map $c\colon \cHor\rightarrow \uHom(\Lambda^2 V, V).$
Now consider two distinguished squares over $f$ with horizontal parts $h$ and $h'.$
As $\cHor\rightarrow P$ is a principal $\uHom(V, \mathfrak{g})$-bundle
there is a unique map $S_{(f, h'), (f, h)}\colon \underline{V}_S\rightarrow \underline{\mathfrak{g}}_P$
over $f$ such that $h' = h + \mathcal{V}\circ S_{(f, h'), (f, h)}.$ 
By adjointness this can be viewed as a map $S_{(f, h'), (f, h)}'\colon S\rightarrow \uHom(V, \mathfrak{g}).$
Then, by Lemmas \ref{lem: infinitesimal invariace} and \ref{lem: dtheta inner product}, we have that for any two sections $v, w\colon S\rightarrow \underline{V}_S$
\begin{align*}
 c(f, h')( v\wedge w) &- c(f, h)(v\wedge w) = d\vartheta \circ h'(v)\wedge h'(w) - d\vartheta \circ h(v)\wedge h(w)\\
                                        & = d\vartheta \circ (h'(v)- h(v))\wedge h'(w) + d\vartheta\circ h(v)\wedge (h'(w)- h(w))\\
					& = d\vartheta \circ (\mathcal{V}\circ S_{(f, h'), (f, h)}(v))\wedge h'(w) + d\vartheta \circ h(v)\wedge (\mathcal{V}\circ S_{(f, h'), (f, h)}(w))\\
					& = -S_{(f, h'), (f, h)}(v)(\vartheta(h'(w))) - d\vartheta \circ (\mathcal{V}\circ S_{(f, h'), (f, h)}(w))\wedge h(v)\\
					& = -S_{(f, h'), (f, h)}(v)(\vartheta(h'(w))) + S_{(f, h'), (f, h)}(w)(\vartheta(h(v)))\\
					& = -S_{(f, h'), (f, h)}(v)(w_f) + S_{(f, h'), (f, h)}(w)(v_f).
\end{align*}

In other words,
\begin{gather*}
c'(f, h') - c'(f, h) = -2 \partial S_{(f, h), (f, h')}'
\end{gather*}
and if we let $\uHom(V, \mathfrak{g})$ act on $\uHom(\Lambda^2V, V)$
via  $(-2)\partial$, then
$c\colon \cHor\rightarrow \uHom(\Lambda^2 V, V)$ is $\uHom(V, \mathfrak{g})$-equivariant.
Now, we have the exact sequence
\[
\xymatrix{
0 \ar[r]& \mathfrak{g}^{(1)} \ar[r]& \uHom(V, \mathfrak{g}) \ar[r]^-\partial & \uHom(\Lambda^2 V, V) \ar[r] & H^{0, 2}(V, \mathfrak{g}) \ar[r] & 0.
}
\]
Consequently, any splitting $s$ of $\im(\partial)\rightarrow \uHom(\Lambda^2V, V)$
gives rise to an equivariant map $\cHor\rightarrow \im(\partial)$
and Lemma \ref{lem: reduction} applied to the short exact sequence
\[
\xymatrix{
0 \ar[r]& \mathfrak{g}^{(1)} \ar[r]& \uHom(V, \mathfrak{g}) \ar[r]^-\partial & \im(\partial) \ar[r] &  0
} 
\]
shows that $\cHor \rightarrow P\times \im(\partial)$ is a principal $\mathfrak{g}^{(1)}$-bundle.
Finally, by pulling back along $P\times \{0\}\rightarrow P\times \im(\partial)$
one obtains the \emph{first prolongation} $P^{(1)}\rightarrow P$, a reduction of $L(P)$
to $\mathfrak{g}^{(1)}$
which consists of those compatible horizontal frames with torsion contained in $C := \ker (s).$

The higher prolongations are now defined inductively: $P^{(i+1)} := (P^{(i)})^{(1)}.$
Setting $\mathfrak{g}^{(-1)} := V$ and $\mathfrak{g}^{(0)} := \mathfrak{g}$, we arrive at
the following inductive description of $\mathfrak{g}^{(k)}$ for $k\geq 1$:
\[
\mathfrak{g}^{(k)} = \{X\in \underline{\mathrm{Hom}}(\mathfrak{g}^{(-1)}, \mathfrak{g}^{(k-1)})\mid X(v)(w) = (-1)^{|v||w|} X(w)(v)\ \text{for all homog.}\ v,\ w\}.
\]

By inspection, we have
\[
(\mathfrak{g}^{(1)})_{\bar{0}}\subseteq (\underline{\mathrm{Hom}}(V, \mathfrak{g})_{\bar{0}})^{\mu}\subseteq \underline{\mathrm{Hom}}(V, \mathfrak{g})_{\bar{0}},
\]
i.e. any $f\in ({\mathfrak{g}^{(1)}})_{\bar{0}}$ satisfies $f({V_{\C}})\subseteq \mathfrak{g}_{\C}.$
This implies that $P^{(k)}\subseteq L^\mu(P^{(k-1)}).$

\subsubsection{The real prolongation}
The prolongations $P^{(k+1)}\rightarrow P^{(k)}$ defined so far
only provide reductions of $L^{\mu}(P^{(k-1)}).$
To prove representability for the functor of automorphisms
of a $G$-structure of finite type, we
need to single out the \emph{real prolongation} which provides a reduction of $L^{\mu}(P^{(k-1)})_\R.$
For this to be possible, we need to impose a condition on the $G$-structure.

To that end, consider the subspaces
\[
\xymatrix{
(\underline{\mathrm{Hom}}(V, \mathfrak{g})_{\bar{0}})^{\mu}_\R\subseteq (\underline{\mathrm{Hom}}(V, \mathfrak{g})_{\bar{0}})^{\mu}\subseteq \underline{\mathrm{Hom}}(V, \mathfrak{g})_{\bar{0}}
}
\]
consisting of even linear maps $f$
satisfying $f(V_\C)\subseteq \mathfrak{g}_\C$
or $f(V_\R, V_\C)\subseteq (\mathfrak{g}_\R, \mathfrak{g}_\C),$
respectively.

Recall the bundle of compatible horizontal frames with the map
$\mathcal{CH}\rightarrow P \times \mathrm{im}(\partial).$
One readily constructs
$(\mathcal{CH}_0)^{\mu}_\R\subseteq (\mathcal{CH}_0)^{\mu}\subseteq \mathcal{CH}_0$
with structure groups $(\underline{\mathrm{Hom}}(V, \mathfrak{g})_{\bar{0}})^{\mu}_{(\R)}.$
Pullback along the inclusion $P_0\times \{0\}\rightarrow P_0\times \im(\partial)_{\bar{0}}$
yields $(P_0^{(1)})^\mu_\R\subseteq (P^{(1)}_0)^\mu\subseteq P^{(1)}_0$
with structure groups given by the pullback
\[
 \xymatrix{
(\mathfrak{g}^{(1)})_{\bar{0}, \R}\ar[r]\ar[d]     & (\underline{\mathrm{Hom}}(V, \mathfrak{g})_{\bar{0}})^{\mu}_{(\R)}\ar[d]\\
(\mathfrak{g}^{(1)})_{\bar{0}} \ar[r] & \underline{\mathrm{Hom}}(V, \mathfrak{g})_{\bar{0}},
}
\]
where once again, it is understood that the quantities in parentheses are only present
for the case of $(P_0^{(1)})^\mu_\R.$
Inductively, we obtain
$(P_0^{(k)})^\mu_\R\subseteq (P^{(k)}_0)^\mu\subseteq P^{(k)}_0$
with structure groups given by the pullback
\[
 \xymatrix{
(\mathfrak{g}^{(k)})_{\bar{0}, \R}\ar[d]\ar[r]     & (\underline{\mathrm{Hom}}(V, \mathfrak{g}^{(k-1)})_{\bar{0}})^{\mu}_{(\R)}\ar[d]\\
(\mathfrak{g}^{(k)})_{\bar{0}} \ar[r] & \underline{\mathrm{Hom}}(V, \mathfrak{g}^{(k-1)})_{\bar{0}}.
}
\]

\begin{Def}
A $G$-structure is called \emph{admissible} if, for all $k\geq 0$,
$(\mathfrak{g}^{(k)})_{\bar{0}, \R}$ defines a mixed structure
on $(\mathfrak{g}^{(k)})_{\bar{0}}.$
\end{Def}

Assume now that the $G$-structure is admissible.
Since $(\mathfrak{g}^{(1)})_{\bar{0}}\subseteq (\underline{\mathrm{Hom}}(V,\mathfrak{g})_{\bar{0}})^\mu$,
we have that $(P_0^{(1)})^\mu = P_0^{(1)}$
and the structure group of $(P^{(1)}_0)^\mu_\R = P^{(1)}_{0, \R}$
is by definition $(\mathfrak{g}^{(1)})_{\bar{0}, \R}.$
Pulling back $r^*P^{(1)}_{0, \R}\rightarrow r^*P^{(1)}_0$ along $P^{(1)}\rightarrow r^* P^{(1)}_0$
gives the functor
$P^{(1)}_\R$, which is representable in view of Lemma \ref{lem: real forms of principal bundles}
and the assumption on the $G$-structure.
All in all, this yields the real prolongation:
\[
 \xymatrix{
\cdots \ar[r] &P_\R^{(k)}\ar[r] & P_\R^{(k-1)} \ar[r] & \cdots \ar[r] & P_\R^{(1)}\ar[r] & P_\R^{(0)}=P \ar[r] & M.
}
\]
The structure group of the $k$th real prolongation will be denoted by $G_\R^{(k)}$.

\section{Automorphisms of $G$-structures}
\label{sec: Automorphisms of G - structures}

The main object of study in this paper is the functor of automorphisms of a $G$-structure,
which we presently define. 

\subsection{The functor of automorphisms of a $G$-structure}
Let $M$ be a mixed supermanifold.
An automorphism $f\colon S\times M\rightarrow S\times M$ over $S$ is called an \emph{$S$-family
of automorphisms of $M.$} Such morphisms assemble to a functor $\underline{\mathrm{Diff}}(M)$
given by
\[
 \underline{\mathrm{Diff}}(M)(S) = \{f\colon S\times M\rightarrow S\times M\mid f\ \text{an $S$-family
of automorphisms of $M$}\}.
\]

Moreover, for any Lie supergroup $G$ and any principal $G$-bundle $P\rightarrow M$,
we let $\underline{\mathrm{Diff}}(P)^G\subseteq \underline{\mathrm{Diff}}(P)$ be the subfunctor of equivariant automorphisms, i.e.
\[
 \underline{\mathrm{Diff}}(P)^G(S) = \{f\in \underline{\mathrm{Diff}}(P)(S)\mid f\ G\text{-equivariant}\}.
\]

Note that if $P$ is a $G$-structure, then inducing up from $G$ to $GL(V)$ gives a map
$\underline{\mathrm{Diff}}(P)^G\rightarrow \underline{\mathrm{Diff}}(L(M))^{GL(V)}$
and, moreover, the differential induces an inclusion of functors
$L(-)\colon \underline{\mathrm{Diff}}(M) \rightarrow \underline{\mathrm{Diff}}(L(M))^{GL(V)}.$

\begin{Def}
The functor of automorphisms of a $G$-structure $P$ on $M$ is defined to be the pullback
\[
\xymatrix{
 \underline{\mathrm{Aut}}(P) \ar[r]\ar[d] & \underline{\mathrm{Diff}}(M)\ar[d]\\
 \underline{\mathrm{Diff}}(P)^G \ar[r] & \underline{\mathrm{Diff}}(L(M))^{GL(V)}.
}
\]
An $S$-point of $\underline{\mathrm{Aut}}(P)$ is called an
\emph{$S$-family of automorphisms of $P$}.
\end{Def}

\begin{Def}
A homogeneous vector field
$\mathcal{O}_{M}\rightarrow ({p_S}_0)_*\mathcal{O}_{S\times M}$
along $p_{S}\colon S\times M\rightarrow M$ is called an
\emph{$S$-family of infinitesimal automorphisms of $P$} if
the induced vector field
$\mathcal{O}_{L(M)}\rightarrow ({p_S}_0)_*\mathcal{O}_{S\times L(M)}$
extends to $\mathcal{O}_{P}\rightarrow ({p_S}_0)_*\mathcal{O}_{S\times P}.$
For $S = *$ this yields the Lie superalgebra
$\mathfrak{aut}(P)\subseteq \Gamma(\mathcal{T}_M)$
of infinitesimal automorphisms of $P.$
The even part has a real subalgebra defined
by
$\mathfrak{aut}(P)_{\bar{0}, \R}:=\mathfrak{aut}(P)_{\bar{0}}\cap \Gamma(\mathcal{T}_{M, \R})$
\end{Def}
\begin{Rmk}
There is no reason for $\mathfrak{aut}(P)_{\bar{0}, \R}$
to be a mixed real form or even a real form of $\mathfrak{aut}(P)_{\bar{0}}.$
For instance, on a purely odd supermanifold all vector fields are real.
The latter would be a necessary condition for the automorphism group to be
representable by a Lie supergroup.
For this reason automorphism groups of $G$-structures are generically
mixed supermanifolds.
\end{Rmk}

In analogy with Lemma \ref{lem: family of automorphisms from group action}, one sees that
any $\varphi\in\underline{\mathrm{Diff}}(M)(\A(\C^{0|n})\times T)$, where $T$ is a mixed
supermanifold, can be uniquely written as
\[
 \varphi^\sharp = \prod_{k = n}^1 \Big(1 + \sum_{k\in I\subseteq \{1, \ldots, k\}} \eta^I X_I\Big) \cdot \varphi_0^\sharp
\]
where $X_I$ are vector fields along $p_{T}$ of degree $|I|$ and $\varphi_0\in\underline{\mathrm{Diff}}(M)(T).$
\begin{Lemma}
\label{lem: family of automorphisms}
Consider $\varphi\in \underline{\mathrm{Diff}}(M)(\A(\C^{0|n})\times T).$
Then $\varphi\in \underline{\mathrm{Aut}}(P)(\A(\C^{0|n})\times T)$ if and only if
$\varphi_0\in \underline{\mathrm{Aut}}(P)(T)$ and all $X_I$ are $T$-families of
infinitesimal automorphisms of $P.$
\end{Lemma}
\begin{proof}
The condition is clearly sufficient.
So, assume that $\varphi$ is an $\A(\C^{0|n})\times T$-family of automorphisms
of $P$. Then $\varphi_0$ is a T-family of such automorphisms since it is obtained by restricting along the inclusion
$T\rightarrow \A(\C^{0|n})\times T.$
Now one proceeds by induction on
$n$ to show that all $X_I$ are infinitesimal automorphisms of $P.$
\end{proof}

\subsection{Prolongation of automorphisms of $G$-structures}
\begin{Prop}
Let $P$ be an admissible $G$-structure on the
mixed supermanifold $M.$
There is a natural inclusion of group-valued functors $\uAut(P)\rightarrow \uAut(P^{(1)}_\R).$
\end{Prop}
\begin{proof}
This follows by repeatedly applying the universal property of the pullback
in the construction of $P^{(1)}_\R.$
\end{proof}

\subsection{The automorphisms of a $\{1\}$-structure}
\label{subsec: The automorphisms of a 1 structure}
We now come to the issue of representability of $\uAut(P).$
Before proceeding to higher order $G$-structures we need to treat the simplest
case $G = \{1\}.$
Then a $G$-structure is simply a parallelization 
$\Phi \colon \underline{V_\R}_M\rightarrow TM_\R.$
Such a $\Phi$ induces an even real vector field on $M\times \A(V_\R):$
\[
 \xymatrix{
 Z\colon M\times \A(V_\R)\ar[r] & TM_\R \times \A(V_\R) \ar[r] & T(M\times \A(V_\R))_\R
 }
\]
and
$\underline{\mathrm{Aut}}(\Phi)(S)$ consists of those 
automorphisms making the diagram
\[
 \xymatrix{
S\times M\times \A(V_\R)\ar[r]^-{S\times \Phi} \ar[d]^{f\times V_\R} & S\times TM_\R\ar[d]^{df}\\
S\times M\times \A(V_\R)\ar[r]^-{S\times \Phi} & S\times TM_\R
}
\]
commutative.

We first show that $i^*\uAut(\Phi)$ is representable.
To that end, we endow
$\mathrm{Aut}(\Phi)_0 := \underline{\mathrm{Aut}}(\Phi)(*)$
with the structure of a Lie group acting on $M.$

Recall that there is a forgetful functor sending
a mixed manifold to its underlying smooth manifold.
(We prove in Section \ref{subsec: the underlying sm of a mixed sm} that such a functor does not exist for mixed supermanifolds.)
Consider the underlying parallelization $\Phi_0\colon M_0\times \A((V_\R)_{\bar{0}})\rightarrow T(M_0)_\R$
and its underlying smooth morphism $\Phi_0^{sm}\colon M_0^{sm}\times (V_\R)_{\bar{0}}\rightarrow TM_0^{sm}.$
In order to define a topology on $\Aut(\Phi)_0$, we need the following fact.

\begin{Lemma}
\label{lem: underlying map injective}
The forgetful map $\Aut(\Phi)_0\rightarrow \Aut(\Phi_0)$, $s\mapsto s_0$, is injective.
\end{Lemma}
\begin{proof}
Deferred to Section \ref{subsec: prfs lems}.
\end{proof}
Moreover, we have that $\Aut(\Phi_0)\subseteq \Aut(\Phi_0^{sm})$
are precisely the elements which preserve the mixed structure on $M_0^{sm}.$

By a result of Kobayashi \cite{Kobayashi}, any point $x\in M_0$ gives rise to a closed injection
\[
 \Aut(\Phi_0^{sm})\rightarrow M_0^{sm},\ s\mapsto s(x)
\]
and with this topology, $\Aut(\Phi_0^{sm})$
is a Lie group such that the evaluation map
$a_0^{sm}\colon \Aut(\Phi_0^{sm})\times M_0^{sm}\rightarrow M_0^{sm}$
is smooth \cite{Ballmann}.
This topology is the coarsest such that
for all $f\in \Gamma(\mathcal{O}_{M_0^{sm}}),$ the map
$\Aut(\Phi^{sm}_0)\rightarrow \Gamma(\mathcal{O}_{M_0^{sm}}),\ s\mapsto s^\sharp(f)$,
is continuous, where $\Gamma(\mathcal{O}_{M_0^{sm}})$ is considered as a Fr\'echet space with respect to the family of seminorms
$|f|_{K, \partial} = \mathrm{sup}_K |\partial f|$, $K\subseteq M_0$ compact, $\partial$ differential operator.
In this topology, $s_n \rightarrow s$ if and only if
$s_n^\sharp(f)\rightarrow s^\sharp(f)$ in
$\Gamma(\mathcal{O}_{M_0^{sm}})$ for all $f\in \Gamma(\mathcal{O}_{M_0^{sm}}).$

Being mixed is a closed condition (locally equations of the form $\partial_{\overline{z}} s^\sharp (f) = 0$
for all $f\in \mathcal{O}_M$), hence $\Aut(\Phi_0)\subseteq \Aut(\Phi_0^{sm})$ is closed.
Then we get a Lie group $\Aut(\Phi)_0^{sm}\subseteq \Aut(\Phi_0)$, in view of the following lemma.
\begin{Lemma}
\label{lem: forgetful cont and closed}
The subspace $\Aut(\Phi)_0\subseteq \Aut(\Phi_0)$
is closed.
The topology on $\Aut(\Phi)_0$ is such that
$s_n\rightarrow s$ implies that for all pairs of coordinate charts $U$, $V$ such that
$s_n(U)\subseteq V$ for all $n$ large enough,
all the coefficients in the Taylor expansion of $s_n^\sharp(f)$, $f\in \Gamma(\mathcal{O}_{M}|_{V})$, with respect
to the odd coordinates converge in $\mathcal{O}_{M_0^{sm}}(U_0).$
 \end{Lemma}
\begin{proof}
Deferred to Section \ref{subsec: prfs lems}.
\end{proof}

In particular, we have an action $a'_0\colon \Aut(\Phi)_0^{sm}\times M_0\rightarrow M_0$
and this is a mixed map since it is so pointwise.
\begin{Lemma}
\label{lem: action}
The map $a'_0$ extends to the action \[
\xymatrix{
 a'^\sharp\colon \mathcal{O}_{M}\ar[r]& \mathcal{O}_{\Aut(\Phi)_0^{sm}\times M},\  f\mapsto (s\mapsto s^\sharp(f)).
}
\]
\end{Lemma}
\begin{proof}
Deferred to Section \ref{subsec: prfs lems}.
\end{proof}

As explained in Section \ref{subsec: flows of even real vector fields},
even real vector fields have unique maximal flows.
Using this, the action above and the description of the topology on $\Aut(\Phi)_0^{sm}$,
one obtains an isomorphism
\[
\mathrm{Lie}_{\R}(\Aut(\Phi)_0^{sm}) \cong \mathfrak{aut}(\Phi)^c_{\bar{0}, \R} :=  \{X\in \Gamma(\mathcal{T}_{M, \R})_{\bar{0}}\mid [X, Z] = 0, X\ \mathrm{complete}\}\subseteq \mathfrak{aut}(\Phi).
\]

Then $\C$-linearization yields a Lie algebra morphism
\[
\xymatrix{
 \C\otimes \mathfrak{aut}(\Phi)_{\bar{0}, \R}^{c} \ar[r]& \Gamma(\mathcal{T}_{M})_{\bar{0}}
}
\]
and the kernel is of the form
\[
\widetilde{\mathfrak{aut}}(\Phi)_{\bar{0}, \C}^c:=\{1\otimes iv- i\otimes v\mid v\in \mathfrak{aut}(\Phi)_{\bar{0}, \C}^c\}
\]
for a complex invariant ideal $\mathfrak{aut}(\Phi)_{\bar{0}, \C}^c\subseteq \mathfrak{aut}(\Phi)_{\bar{0}, \R}^c.$
This yields the mixed structure $\Aut(\Phi)_0$ on $\Aut(\Phi)_0^{sm}$
and, on general grounds, the quotient
\[
(\C\otimes{\mathfrak{aut}(\Phi)}_{\bar{0},\R}^{c})/{\widetilde{\mathfrak{aut}}}(\Phi)_{\bar{0}, \C}^c := \mathfrak{aut}(\Phi)^{c, d}_{\bar{0}} \subseteq \mathfrak{aut}(\Phi)_{\bar{0}}
\]
is the Lie algebra of left-invariant derivations of $\mathcal{O}_{\Aut(\Phi)_0}.$
It is the algebra of  \emph{complete decomposable infinitesimal
automorphisms} in the sense that any of its elements can be written as the sum $v + iw$ of 
complete real vector fields $v$ and $w.$
Moreover, with this structure $a\colon \Aut(\Phi)_0\times M\rightarrow M$
is a mixed morphism, by Lemma \ref{lem: action underlying real Lie group and mixed group}.

Finite-dimensionality of the full algebra of infinitesimal
automorphisms is ensured by the following lemma.
\begin{Lemma}
\label{lem: aut Phi finite dimensional}
Assume that $M_0$ is connected.
For every $p\in M_0$, evaluation
$\mathfrak{aut}(\Phi)\rightarrow {T_pM}$, $X\mapsto X(p)$, is
injective.
If $M_0$ is not connected, the analogous statement holds
true if one chooses one point for each connected component.
\end{Lemma}
\begin{proof}
Deferred to Section \ref{subsec: prfs lems}.
\end{proof}

Moreover, the conjugation action of  $\Aut(\Phi)_0$
on $\Gamma(\mathcal{T}_M)$ restricts to an action on
$\mathfrak{aut}(\Phi)$
and the differential of this representation is simply the restriction of the
adjoint representation
\[
\xymatrix{
{\mathfrak{aut}(\Phi)}_{\bar{0}}^{c, d}\times \mathfrak{aut}(\Phi) \ar[r]& \mathfrak{aut}(\Phi).
}
\]

The following result shows that $\Aut(\Phi)_0$ has the correct topology
and mixed structure.

\begin{Prop}
The functors $i^*\uAut(\Phi)$ and $\mathsf{M}^\mu(-, \Aut(\Phi)_0)$
are naturally isomorphic.
\end{Prop}
\begin{proof}
Given a map $T_0\rightarrow \Aut(\Phi)_0$, the action of the group
yields a map $T_0\times M\rightarrow T_0\times M.$
Conversely, take an element $f\colon T_0\times M\rightarrow T_0 \times M$ in $\uAut(\Phi)(T_0).$
The obvious candidate $\tilde{f}\colon T_0\rightarrow \Aut(\Phi)_0$ is a smooth mixed map since
the composition
\[
 T_0\rightarrow \Aut(\Phi)_0\rightarrow M_0
\]
with evaluation at some $m\in M_0$
equals $f_0(-, m_0)$ which is smooth and mixed.
\end{proof}

\subsection{The automorphisms of a $G$-structure of finite type}

\begin{Def}
An admissible $G$-structure is of \emph{finite type} if there exists a $k\geq 0$ such that $G_\R^{(k+l)} = \{1\}$
for all $l\geq 0.$
\end{Def}

The main theorem is as follows:
\begin{Thm}
\label{thm: main}
Suppose $P\rightarrow M$ is an admissible
$G$-structure of finite type.
Then $i^* \uAut(P)$ is representable
and its (real) Lie algebra consists of the complete real infinitesimal
automorphisms of $P$, denoted by $\mathfrak{aut}(P)_{\bar{0},\R}^{c}.$

Moreover, $\mathfrak{aut}(P)$ is finite-dimensional
and the functor $\uAut(P)$ is representable if 
and only if $\mathfrak{aut}(P)_{\bar{0}}^{c, d} = \mathfrak{aut}(P)_{\bar{0}}.$
\end{Thm}

\begin{proof}
We choose $k\geq 0$ such that $G_\R^{(k+l)} = \{1\}$ for all $l\geq 0.$
Then we have an embedding $\uAut(P)\rightarrow \uAut(P^{(k)}_\R).$
Hence, $\mathfrak{aut}(P)$ is finite-dimensional in view of Lemma \ref{lem: aut Phi finite dimensional}.
Let $\Phi$ be the given parallelization of the real tangent bundle of $P^{(k-1)}_\R.$

We show that the inclusion
\[
\uAut(P)(*)\subseteq \uAut(\Phi)(*) =  \Aut(\Phi)_0^{sm}
\]
is closed.
Recall that the topology on $\Aut(\Phi)_0^{sm}\subseteq (P_\R^{(k-1)})_0^{sm}$ is such that $s_n\rightarrow s$
implies that locally all $s_n^\sharp(f)$, $f\in \Gamma(\mathcal{O}_{P_\R^{(k-1)}}|_{V})$, converge
in the closed subspace
\[
\mathcal{O}_{P_\R^{(k-1)}}(U_0)\cong \bigoplus \mathcal{O}_{(P_\R^{(k-1)})_0}(U_0)\subseteq \bigoplus\mathcal{O}_{(P_\R^{(k-1)})_0^{sm}}(U_0),
\]
where the number of summands is $2^d$, $d$ denoting the odd dimension of $P^{(k-1)}_\R.$

Now assume $s_n\in \uAut(P)(*)$ and $s_n^{(k)}\rightarrow \tilde{s}.$
From the construction of the prolongation, it is clear that one obtains a diffeomorphism $s\colon M\rightarrow M$
with $k$th prolongation $s^{(k)}$ equal to $\tilde{s}.$
From equivariance it now follows that $s$ is actually in $\uAut(P)(*).$

Next, assume that
the action $\Aut(P^{(i+1)}_\R)_0^{sm}\times {P^{(i)}_\R}\rightarrow {P^{(i)}_\R}$
is smooth.
Restricted to $\Aut(P^{(i)}_\R)_0^{sm}$,
it is pointwise equivariant, hence it is itself
equivariant
and thus descends to an action on $P^{(i-1)}_\R.$
This action gives the identification of the Lie algebra of $\Aut(P)_0^{sm}$ with
$\mathfrak{aut}(P)^{c}_{\bar{0}, \R}$,
and the mixed structure is now defined as in the case of $\Aut(\Phi)_0.$
Then the action just defined refines to an action $\Aut(P)_0\times M\rightarrow M$
by Lemma \ref{lem: action underlying real Lie group and mixed group}
and, using this, similar as in the situation of the automorphisms
of a parallelization, one deduces that $i^* \uAut(P)\cong \Aut(P)_0.$

Clearly, if $\uAut(P)$ is representable, then $\mathfrak{aut}(P)$ can only
consist of complete and decomposable vector fields.
Conversely,
if $\mathfrak{aut}(P)_{\bar{0}}^{c, d} = \mathfrak{aut}(P)_{\bar{0}}$,
then
\[
\Aut(P) = (\mathfrak{aut}(P), \Aut(P)_0)
\]
forms a mixed super pair. The action defines a map
$\SMan^\mu(-,\Aut(P))\rightarrow \underline{\mathrm{Diff}}(M)$,
and in view of Lemma \ref{lem: family of automorphisms from group action},
it factors locally through an isomorphism
to $\uAut(P).$
Hence, it factors globally as an isomorphism $\SMan^{\mu}(-, \Aut(P))\cong \uAut(P).$ 
\end{proof}

\subsection{Proofs of Lemmas \ref{lem: underlying map injective}, \ref{lem: forgetful cont and closed}, \ref{lem: action}, and \ref{lem: aut Phi finite dimensional}}
\label{subsec: prfs lems}
\begin{proof}[Proof of Lemma \ref{lem: underlying map injective}]
Let $s\in \Aut(\Phi)_0$ be such that $s_0 =\id.$
In order to see that this implies $s = \id$, we consider, for $k\geq 1$, the restriction of $s$ to the
$(k-1)$th infinitesimal
neighbourhood
\[
\xymatrix{
(s^{(k-1)})^\sharp\colon \mathcal{O}_M/\mathcal{J}^k \ar[r]& (s_0)_*\mathcal{O}_M/\mathcal{J}^k.
}
\]
We have $(s^{(0)})^\sharp = s_0^\sharp = \id.$

Now, we choose a homogeneous basis
$\{v_1, \ldots, v_n, v_{n+1}, \ldots, v_{n+m}\}$ of $V_\R$
and local coordinates
$\{q_1, \ldots, q_n, q_{n+1}, \ldots, q_{n+m}\}$
on an open subset $U_0$ containing $m\in M_0.$
Here, the first $n$ (respectively last $m$) entries are assumed to be even
(resp. odd). 
In the given basis
\[
Z_{v_k} = \sum_{l} A_{kl} \partial_{q_l}
\]
for some even invertible matrix $A = (A_{kl})\in GL_{\mathcal{O}_M}(\mathcal{O}_M(U_0)^{n|m}).$
The requirement for $f$ to lie in $\Aut(\Phi)_0$ reads
\[
 Jf = A^{-1}\circ f^\sharp(A)
\]
where $Jf = (\partial_{q_i} f^\sharp(q_j))$
and we denote the natural extension of $f^\sharp$ to matrices by the same symbol.

So assume $(f^{(k-1)})^\sharp = \id.$
We have
\begin{align*}
 Jf + \mathcal{J}^k(U_0)^{(n|m)\times (n|m)} & = A^{-1}\circ f^\sharp(A) + \mathcal{J}^k(U_0)^{(n|m)\times (n|m)}\\
                      & = A^{-1}\circ (f^{(k-1)})^\sharp(A) + \mathcal{J}^k(U_0)^{(n|m)\times (n|m)}\\
                      & = \id_{n|m} + \mathcal{J}^k(U_0)^{(n|m)\times (n|m)},
\end{align*}
and this implies $(f^{(k)})^\sharp = \id.$
\end{proof}

\begin{proof}[Proof of Lemma \ref{lem: forgetful cont and closed}]
Let $\{s_n\}$ be a sequence in $\Aut(\Phi)_0$ such that $\{(s_n)_0\}$ converges
to some $\tilde{s}.$ We have to show that $\tilde s = s_0$ for some suitable $s\in \Aut(\Phi)_0$
and that $s_n$ converges to $s.$
Without loss of generality all $(s_n)_0$ lie in one coordinate chart (in $\Aut(\Phi_0)$)
and since $a_0^{sm}$ is smooth we may choose open subspaces $U$ and $V$
with coordinates $\{p_i\}$ and $\{q_i\}$ respectively
such that every $s_n$ restricts to a map $U\rightarrow V.$
Let us organise the coordinates
into even and odd functions
$\{p_i\} = \{x_i, \eta_j\}$, $\{q_i\} = \{y_i, \xi_j\}.$

In these coordinate charts the condition for $s_n$ to lie in $\Aut(\Phi)_0$
reads
\[
 J(s_n) =  A\cdot s_n^\sharp(B)
\]
for certain invertible matrices $A$ and $B$ where
$J(s_n) = (\partial_{p_i} s_n^\sharp(q_i)).$
Starting from ${s^{(0)}}^\sharp := \tilde{s}^\sharp$,
we inductively define $(s^{(k)})^\sharp\colon \mathcal{O}_M/\mathcal{J}^{k+1}(V_0)\rightarrow \mathcal{O}_M/\mathcal{J}^{k+1}(U_0)$
with reductions $\tilde{s}.$
The construction will be such that the following holds:
We have $(s_n^{(k)})^{\sharp}(f)\rightarrow (s^{(k)})^\sharp(f)$
for all $f\in (\mathcal{O}_M/\mathcal{J}^{k+1})(V_0).$
Here, $(\mathcal{O}_M/\mathcal{J}^{k+1})(U_0)$
is considered as a subspace of $\bigoplus\mathcal{O}_{M_0^{sm}}(U_0)$,
where the number of summands is $2^m.$

The respective lifts will be determined by the Jacobian $J(s^{(k)})$
which naturally has values in matrices of the form
\[ \left( \begin{array}{cc}
\mathcal{O}_M/\mathcal{J}^{k+1} & \mathcal{O}_M/\mathcal{J}^{k}  \\
\mathcal{O}_M/\mathcal{J}^{k+1} & \mathcal{O}_M/\mathcal{J}^{k}\\
\end{array} \right).\] 
There is a projection from $\mathcal{O}_M/\mathcal{J}^{k+1}$-valued matrices
to such matrices. The image of a matrix $A$ will be denoted by $A^\sim.$

Assume that $k$ is even and $(s^{(k)})^{\sharp}$ has been constructed such that
\[
J(s^{(k)}) = (A^{(k)} (s^{(k)})^\sharp B^{(k)})^\sim.
\]
First, we have to set $(s^{(k+1)})^\sharp(q_i) = (s^{(k)})^\sharp(q_i)$ for $q_i$ even.
The odd-odd sector of the Jacobian determines $(s^{(k+1)})^\sharp(q_i)$ for $q_i$ odd:
In fact, it follows that
\begin{align*}
\partial_{\eta_i} (s^{(k+1)})^{\sharp}(\xi_j) & \stackrel{!}{=} (A^{(k+1)} (s^{(k+1)})^\sharp B^{(k+1)})_{ij}^\sim\\
                                            & = (A^{(k)} (s^{(k)})^\sharp B^{(k)})_{ij}\\
                                            & = \lim_n (A^{(k)} (s_n^{(k)})^\sharp B^{(k)})_{ij}\\
                                            & = \lim_n  \partial_{\eta_i} (s_n^{(k+1)})^{\sharp}(\xi_j).
\end{align*}

These derivatives fit
together to give a well-defined $(s^{(k+1)})^\sharp(\xi_j)$
since the different partial derivatives fit together, that is, for any multiindex $I$, $|I| = k+1$, with $\eta_i$, $\eta_{i'}\in I$,
we have
\[
\partial_{I-\{\eta_i\}} \partial_{\eta_i} ((s^{(k+1)})^{\sharp}(\xi_j)) =  \epsilon_{i, i'}\partial_{I-\{\eta_{i'}\}} \partial_{\eta_{i'}} ((s^{(k+1)})^{\sharp}(\xi_j))
\]
since this equality holds for all $s_n.$
With this definition we have $(s^{(k+1)})^\sharp = \lim_n (s_n^{(k+1)})^\sharp$, which ensures
$J(s^{(k+1)}) = (A^{(k+1)} (s^{(k+1)})^\sharp B^{(k+1)})^\sim$
by continuity.

If $k$ is odd and $(s^{(k)})^{\sharp}$ has been constructed in such a way that
\[
J(s^{(k)}) = (A^{(k)} \cdot (s^{(k)})^\sharp B^{(k)})^\sim,
\]
then one can proceed similarly.
There are no changes in the pullbacks of odd coordinates and
the pullbacks of the even coordinates are forced by the respective equation for the odd-even sector of the
Jacobian.
Again, $(s^{(k)})^\sharp = \lim (s_n^{(k)})^\sharp.$
This yields the construction of $s|_U\colon \mathcal{O}_{V}\rightarrow (s_0)_*\mathcal{O}_{U}.$
By uniqueness (Lemma \ref{lem: underlying map injective}),
these $s|_U$ coincide where two coordinates patches overlap,
and so we obtain the desired $s\colon M\rightarrow M.$

The statement concerning the topology is clear from the above considerations.
\end{proof}

\begin{proof}[Proof of Lemma \ref{lem: action}]
Similary as in the preceding lemma, starting from $((a')^{(0)})^\sharp := (a_0')^\sharp$,
we inductively construct
$((a')^{(k)})^\sharp\colon \mathcal{O}_M/\mathcal{J}^{k+1} \rightarrow (a_0')_*\mathcal{O}_{\Aut(\Phi)_0^{sm}\times M}/\mathcal{J}^{k+1}.$
First we choose some neighbourhoods $W\subseteq \Aut(\Phi)_0^{sm}$ and
$U$, $V\subseteq M$
given by coordinates $\{p_i\} = \{x_i, \eta_j\}$ and $\{q_i\} = \{y_i, \xi_j\}$
such that $a_0'$ restricts to
\[
\xymatrix{
 W \times U_0 \ar[r] & V_0.
}
\]

Then, if $A$ and $B$ are as in the proof above, the map $(a')^\sharp$ to be constructed
will be characterized by
\[
 J^{\mathrm{res}}(a') = A (a')^\sharp(B).
\]
where $J^{\mathrm{res}}(a')$ denotes the submatrix $(\partial_{p_i} (a')^{\sharp}(q_j))$
of the Jacobian.
So, assume $((a')^{(k)})^\sharp$ is constructed such that
\[
 J((a')^{(k)}) = (A^{(k)} ((a')^{(k)})^\sharp B^{(k)})^\sim.
\]
Suppose first that $k$ is even.
Looking at the odd-odd sector of the Jacobian gives
\begin{align*}
 \partial_{\eta_i}((a')^{(k+1)})^\sharp(\xi_j) =  (A^{(k)} ((a')^{(k)})^\sharp B^{(k)})_{ij}.
\end{align*}
These fit together since they do so pointwise,
i.e.~after specializing to any element $s\in \Aut(\Phi)_0^{sm}.$
Moreover, the identity for the Jacobian holds true, since it holds true pointwise.
\end{proof}

\begin{proof}[Proof of Lemma \ref{lem: aut Phi finite dimensional}]
We follow \cite[Lemma 2.4]{Ballmann}.
If $X\in\mathfrak{aut}(\Phi)$, then $X_{V_\R} := X\otimes\id_{V_\R}$ is a vector field
on $M\times \A(V_\R)$ which commutes with $Z$ (as is seen in local coordinates).

Let $\Theta^Z$ be the maximal flow of the even real vector field $Z$
(see Theorem \ref{thm: maximal flow}), defined on
$\mathcal{V}\subseteq \R\times M\times \A(V_\R)$,
and consider the composite
${\Theta^Z}' = {\mathop{\mathrm{pr}}}_1 \circ \Theta^Z\colon \mathcal{V}\rightarrow M.$
Note that $\{1\}\times M\times \{0\}\subseteq \mathcal{V}$, so
${\Theta^Z}'(1, -)$ is defined on an open neighbourhood of $M\times \{0\}.$

We have the following:
For all $p\in M_0$ there exists an open neighbourhood $p\in U_0\subseteq M_0$ and and
open subspace $V'\subseteq \A(V_\R)$ such that
for all $q\in U_0$ the map ${\Theta^Z}'(1, q, -)\colon V'\rightarrow M$ is a diffeomorphism onto an open subspace.

Indeed, the map $(\mathop{\mathrm{pr}}_1, {\Theta^Z}'(1,-))$ is defined on an open neighbourhood of $M\times \{0\}$
and its differential at $(p, 0)$ is of the form
\[
\left(\begin{array}{cc}
 1 & 0\\
 * & Z
\end{array}
\right),
\]
which is invertible.

Now, assume $\mathrm{inj}_p^\sharp \circ X = X(p) = 0.$
Choose open subspaces $U\subseteq M$
and $V'\subseteq V$ such that $p\in U_0$ and $0\in V'$
such that
$\varphi := \Theta^Z(1, p, -)\colon V'\rightarrow U$ is an isomorphism.
Then
\begin{align*}
 \varphi^\sharp \circ X & = \mathrm{inj}_p^\sharp\circ \Theta^Z(1, -, -)^\sharp\circ {\mathop{\mathrm{pr}}}_1^\sharp \circ X\\
                        & = \mathrm{inj}_p^\sharp\circ \Theta^Z(1, -, -)^\sharp\circ X_V\circ {\mathop{\mathrm{pr}}}_1^\sharp\\
                        & = \mathrm{inj}_p^\sharp\circ X_V\circ  \Theta^Z(1, -, -)^\sharp\circ {\mathop{\mathrm{pr}}}_1^\sharp\\
                        & = 0,
\end{align*}
where we have used Proposition \ref{prop: commuting flow and vector fields}
in the third line.
Since $\varphi^\sharp$ is invertible, it follows that $X = 0$ on $U.$

This shows that the non-empty closed set $\{p\in M_0| X(p) = 0\}$
is contained in the open subset $\{p\in M_0| X_p = 0\}.$
The converse inclusion holds always, so that both subsets agree
and are open and closed, hence they are all of $M_0$
if $M_0$ is connected.
More generally, the argument shows that $X(p) = 0$ implies
$X = 0$ on the connected component containing $p.$
\end{proof}

\section{$G$-structures of finite type on real supermanifolds}
\label{sec: G - structures of finite type on real supermanifolds}
Results analogous to those obtained in the mixed case hold for real supermanifolds.
Their proofs are simplifications of our previous arguments, so we only briefly comment
on them to provide precise statements for future reference.

A real super vector space is $\Z/2$-graded real vector space $V = V_{\bar{0}}\oplus V_{\bar{1}}.$
The model spaces for real supermanifols are the affine spaces
$\A(V)  = (V_0, C^\infty_{V_0}(-)\otimes \bigwedge V_1^*).$

\begin{Def}
A \emph{real supermanifold} consists of
a locally ringed superspace $M = (M_0, \mathcal{O}_M)$ over $\R$ with
second countable Hausdorff base
that is locally isomorphic to
$\A(V)$ for some real super vector space $V.$
The full subcategory of locally ringed superspaces
over $\R$ with objects real supermanifolds
is denoted by $\SMan_\R.$
\end{Def}

Similarly as in the case of supermanifolds, a real supermanifold has a
frame bundle $L(M)$, which is a principal $GL(V)$-bundle.
In the real category, $GL(V)$ is a real Lie supergroup and so
$L(M)$ is an object in the category of real manifolds.
Furthermore, given a $G$-structure, i.e. a closed subgoup $G\leqslant GL(V)$
and a reduction $P$ of $L(M)$ to $G$,
one can define the prolongation without leaving the real category.

One has a functor $i\colon \Man\rightarrow \SMan_\R$
and similarly as in the case of mixed supermanifolds, one obtains the following result.

\begin{Thm}
Suppose $P\rightarrow M$ is a 
$G$-structure of finite type.
Then $i^* \uAut(P)$ is representable
and its Lie algebra consists of the complete infinitesimal
automorphisms of $P$, denoted by $\mathfrak{aut}(P)_{\bar{0}}^{c}.$
Moreover, $\mathfrak{aut}(P)$ is finite dimensional.
The functor $\uAut(P)$ is representable if 
and only if $\mathfrak{aut}(P)_{\bar{0}}^{c} = \mathfrak{aut}(P)_{\bar{0}}.$
\end{Thm}

\section{Examples of $G$-structures of finite type}
\label{sec: Examples of G - structures of finite type}
\subsection{Riemannian structures on supermanifolds}
In this section, we treat
Riemannian structures on a supermanifold $M$
locally modelled on the mixed super vector space $(V, V_\R, V_\C).$

\subsubsection{Even Riemannian structures}
Consider an even non-degenerate bilinear
form $J\colon V\otimes V\rightarrow \C^{1|0}$
with components $J_i\colon V_{\bar{i}}\otimes V_{\bar{i}}\rightarrow \C$
($i\in\{0, 1\}$).
There is a Lie supergroup $OSp(V, J)$ which represents automorphisms of the trivial vector
bundle endowed with $J$:
\[
 OSp(V, J)(S) = \{f\in GL(V)(S)\mid (S\times J)\circ (f\otimes f) = (S\times J)\}.
\]
\begin{Prop}
$ $
\begin{itemize}
\item[(a)] Reductions of $L(M)$ to $OSp(V, J)$ are in bijective correspondence with even non-degenerate supersymmetric maps of vector bundles
$TM\otimes TM\rightarrow \underline{\C^{1|0}}_M.$
\item[(b)] $OSp(V, J)\leqslant GL(V)$ is of finite type, more precisely $\mathfrak{osp}(V, J)^{(1)} = 0.$
\end{itemize}
\end{Prop}
\begin{proof}
Given an $OSp(V, J)$-structure on $L(M)$, one constructs a metric on $TM$ by declaring the given bases
to be orthonormal. Conversely, given a metric, the orthonormal bases give rise to an $OSp(V, J)$-structure.
This shows the first part.
In order to show the second part, we observe that $\mathfrak{osp}(V, J)$ consists
of those endomorphisms $A\colon V\rightarrow V$ whose homogeneous components $A_i$
satisfy $J(A_iv, w) = -(-1)^{|A_i||v|}J(v, A_iw).$
Using a homogeneous basis $\{v_i\}$, the conditions for $T$ to lie in $\mathfrak{osp}(V, J)^{(1)}$
read $T^{i}_{jk} = (-1)^{|v_i||v_j|} T^{j}_{ik}$ and $T^{i}_{jk} = - (-1)^{|v_j||v_k|} T^{i}_{kj}$,
where we set $T^i_{jk} =J(T(v_i)v_j, v_k).$
Both together imply $T^i_{jk}  = 0.$
\end{proof}

The underlying complex group of $OSp(V, J)$ is the product of the complex groups $O(V_{\bar{0}}, J_0)\times Sp(V_{\bar{1}}, J_1).$
Assume that $J_0$ restricts to a non-degenerate bilinear form $J_{0, \R}\colon (V_{\bar{0}})_\R\otimes (V_{\bar{0}})_\R\rightarrow \R.$
Such a $J$ gives rise to the mixed real form
$OSp(V, J)_\R\rightarrow OSp(V, J)$
with underlying group $O((V_{\bar{0}})_\R, J_{0, \R})\times Sp(V_{\bar{1}}, J_1).$
Moreover, $OSp(V, J)_\R\leqslant GL(V)_\R.$
\begin{Lemma}
The $OSp(V, J)_\R$-structures on $M$ are in bijective correspondence
with
even non-degenerate supersymmetric maps of vector bundles
$TM\otimes TM\rightarrow \underline{\C^{1|0}}_M$
whose restriction to $(TM)_{\bar{0}}\otimes (TM)_{\bar{0}}\subseteq i^*(TM\otimes TM)$
induce a metric $T(M_0)_\R\otimes T(M_0)_\R\rightarrow \underline{\R^1}_{M_0}$
of the same signature as $J_{0, \R}$ on the underlying real manifold $M_0.$
\end{Lemma}
\begin{proof}
 This follows readily from the definition of $OSp(V, J)_\R.$
\end{proof}

From Theorem \ref{thm: main}, we obtain the following result.
\begin{Thm}
Let $M$ be a supermanifold with an $OSp(V, J)_\R$-structure.
If $M_0$ is complete and every Killing vector field is decomposable, then the isometry group functor $\uAut(P)$
is representable.
\end{Thm}

\begin{Rmk}
In the real category the only obstruction for representability
is completeness of the Killing vector fields.
In this setting, an isometry group was constructed by Goertsches \cite{Goertsches}.
(The completeness condition seems to be assumed implicitly.)
Our results in the real case give a rederivation of this result.
\end{Rmk}
\begin{ex}
The isometry group of $V$ with the $OSp(V, J)_\R$ as above is
$OSp(V, J)_\R\ltimes V_\R.$
\end{ex}

\subsubsection{Odd Riemannian structures}

In the super setting, there is an odd analogue of
a Riemannian structure, given by an odd non-degenerate supersymmetric bilinear
form $J\colon V\otimes V\rightarrow \C^{1|0}.$
The Lie supergroup $P(V, J)$ is defined by the functor
\[
 P(V, J)(S) = \{f\in GL(V)(S)\mid (S\times J)\circ (f\otimes f) = (S\times J)\}.
\]

Similar to the even case, one can show the following.
\begin{Prop}
$ $
\begin{itemize}
\item[(a)] The $P(V, J)$-structures on $L(M)$ and the odd non-degenerate supersymmetric maps of vector bundles
$TM\otimes TM\rightarrow \underline{\C^{1|0}}_M$ are in one-to-one correspondence.

\item[(b)] $P(V, J)\leqslant GL(V)$ is of finite type, more precisely $\mathfrak{p}(V, J)^{(1)} = 0.$
\end{itemize}
\end{Prop}

We have $P(V, J)_0 \cong GL(V_{\bar{0}})$, which comes with the mixed real form given by $GL((V_{\bar{0}})_\R)$
and thus gives rise to $P(V, J)_\R\leqslant GL(V)_\R.$

For any $P(V, J)$-structure $P$ on $M$,
we have that $P_0 \cong L(M_0)$ and hence, it admits the real form $P_{0, \R}\cong L(M_0)_\R.$
Now, one easily concludes the following.

\begin{Prop}
$P(V, J)_\R$-structures 
are in one-to-one correspondence with $P(V, J)$-structures.
\end{Prop}

From Theorem \ref{thm: main}, we obtain the following result.
\begin{Thm}
Let $M$ be a supermanifold with a $P(V, J)_\R$-structure.
If $M_0$ is complete and all infinitesimal automorphisms are decomposable, then the isometry group functor $\uAut(P)$
is representable.
\end{Thm}

\subsection{Superization of Riemannian spin manifolds}
Let $(M_0, g_0)$ be a connected pseudo-Riemannian spin manifold
endowed with a $\Spin(V_{\bar{0}})$-structure
\[
\xymatrix{
\rho(M_0)\colon \Spin(M_0)\ar[r] & SO(M_0),
}
\]
where we set $(V_{\bar{0}}, \alpha) = (T_m M_0, g_m)$ for some $m\in M_0.$
Choose any real or complex $Cl(V_{\bar{0}}, \alpha)$- or $Cl(V_{\bar{0}}, \alpha)\otimes \C$-module $V_{\bar{1}}.$

The spinor bundle is the associated bundle $\mathcal{S} = \Spin(M_0)\times^{\Spin(V_{\bar{0}})} V_{\bar{1}}\rightarrow M_0$,
which we endow with the lift of the Levi-Civita connection.
Then $TM_0\oplus \mathcal{S}\rightarrow M_0$ admits a reduction
to $Spin(V_{\bar{0}})\leqslant GL(V_{\bar{0}})\times GL(V_{\bar{1}})$ by means of
$(\rho(M_0), \id)\colon Spin(M_0)\rightarrow SO(M_0)\times \Spin(M_0).$

The spinor supermanifold $M$ associated to this data is obtained by taking the exterior algebra of the dual $\mathcal{S}^*$:
\[
M = (M_0, \Gamma(\textstyle{-}, \bigwedge \mathcal{S}^*)).
\]
It is a real supermanifold or a supermanifold depending on whether $V_{\bar{1}}$ is chosen to be real or complex.
Any vector field on $M_0$ can be extended to $M$
by means of the dual connection on $\mathcal{S}^*$, $X\mapsto \nabla_X$,
and, furthermore, dual spinors can be contracted with spinors.
This yields an inclusion $\iota\colon TM_0\oplus \Pi\mathcal{S}\rightarrow TM$
and hence a $\Spin(V_{\bar{0}})$-structure $P_{\Spin(V_{\bar{0}})}\subseteq L(M)_\R.$

Any $\Spin(V_{\bar{0}})$-submodule $\mathcal{W}\subseteq \uHom(V_{\bar{0}}, V_{\bar{1}})$ gives rise to a mixed Lie supergroup
$Spin(V_{\bar{0}})\ltimes \mathcal{W}\leqslant GL(V)_\R.$
Consequently, by inducing up, any such $\mathcal{W}$ gives rise to a $\Spin(V_{\bar{0}})\ltimes \mathcal{W}$-structure
on $M$:
\[
 P_{Spin(V_{\bar{0}})\ltimes \mathcal{W}} := P_{Spin(V_{\bar{0}})}\times^{Spin(V_{\bar{0}})} (Spin(V_{\bar{0}})\ltimes \mathcal{W}).
\]

A particular choice is
\[
\mathcal{W} = \{f_s\colon V_{\bar{0}}\rightarrow V_{\bar{1}}\mid s\in V_{\bar{1}},\ f_s(v_0) = v_0 \cdot s\}.
\]

\begin{Prop}
For this choice of $\mathcal{W}$, $\Spin(V_{\bar{0}})\ltimes \mathcal{W}\leqslant GL(V)$
is of finite type, provided that $\dim M\geq 3.$
\end{Prop}
\begin{proof}
 After choosing an orthonormal basis $\{e_i\}$ of $V_{\bar{0}}$, everything boils down to
 \[
  e_i \cdot s_j = e_j \cdot s_i
 \]
for all $i$ and $j$ and certain $s_i\in V_{\bar{1}}$, which implies $s_j = 0$ if $\dim M\geq 3$:
We have $s_k = -(e_l, e_l) e_l e_k s_l.$
On one hand, if $k$, $l$ and $j$ are such that $l\neq j$ and $l\neq k$
we have
\begin{align*}
 s_k & = -(e_l, e_l) e_l e_k s_l\\
     & = -(e_l, e_l) e_l e_k (- (e_j, e_j) e_j e_l s_j)\\
     & = -(e_j, e_j) e_k e_j s_j
\end{align*}
On the other hand
\[
 s_k = -(e_j, e_j) e_j e_k s_j.
\]
So, if in addition $k\neq j$ (hence all three are different), then
\[
 e_j e_k s_j = 0,
\]
so that we finally arrive at $s_j = 0.$
\end{proof}

\begin{Rmk}
By a theorem of Cort\'es et al.~\cite{Cortes et al}, the vector field $\iota(s)$ associated
with a spinor 
gives rise to an infinitesimal automorphism of $P_{\Spin(V_0)\ltimes \mathcal{W}}$
if and only if $s$ is a twistor spinor, i.e.~there exists a spinor $\tilde{s}$ such that for all $X$ we have $\nabla_X s = X\cdot \tilde{s}.$
\end{Rmk}

\section{Appendix}
\label{sec: Appendix}
\subsection{Non-existence of a forgetful functor $\SMan^\mu\rightarrow \SMan$}
\label{subsec: Non-existence of a forgetful functor}

A mixed manifold $M$ has an underlying manifold $M^{sm}$
which comes with a functorial map $M^{sm}\rightarrow M.$
For an affine space $M = \A(V, V_\R, V_\C)$, the assignment is simply given by
setting $M^{sm} = \A(\C\otimes V_\R, V_\R, 0)$,
and the map $M^{sm}\rightarrow M$ is induced by the map $\C\otimes V_\R\rightarrow V.$
We show that the analogous statement fails in the category of mixed supermanifolds.
This is not surprising insofar as there does not even exist a forgetful functor from complex supermanifolds
to supermanifolds \cite{Witten}.
A by-product of the argument is a proof that there is no functorial way to split even complex functions
on supermanifolds into two even real functions (Proposition \ref{prop: non-existence of functorial lift from complex functions to two real functions}).

Let $(V, V_\R, V_\C)$ be a mixed super vector space.
The natural choice for the underlying supermanifold is given by
the affine space associated with the super vector space
$u(V, V_\R, V_\C) = (\C\otimes (V_\R)_{\bar{0}}\oplus V_{\bar{1}}, V_\R, V_{\bar{1}}).$
The natural choice for the map
\[
\xymatrix{
 \epsilon_{(V, V_\R, V_\C)}\colon \A(u(V, V_\R, V_\C))\ar[r] & \A(V, V_\R, V_\C)
}
\]
is induced by the
$\C$-linearization of the inclusion $(V_\R)_{\bar{0}}\rightarrow V_{\bar{0}}$
and the identity on $V_{\bar{1}}.$
Note that $u^2 = u.$
However, these natural choices do not assemble to a forgetful functor
from mixed supermanifolds to supermanifolds:
\label{subsec: the underlying sm of a mixed sm}
\begin{Prop}
\label{prop: non-existence of forgetful functor}
There is no functor $F\colon \SMan^\mu\rightarrow \SMan$
such that the following two conditions hold:
\begin{itemize}
\item[(a)]$F(\A(V, V_\R, V_\C)) = \A(u(V, V_\R, V_\C))$ and $F(\A(\epsilon_{(V, V_\R, V_\C)})) = \id_{\A(u(V, V_\R, V_\C))}.$
\item[(b)] $F|_{\SMan} = \id_{\SMan}.$
\end{itemize}
\end{Prop}
\begin{proof}
Assume that such a functor $F$ existed.
Consider $\A(\C)$ and $\A(\R^2)$
with their standard monoid structure.
Then we had a commutative square
\[
 \xymatrix{
\A(\R^2)\times \A(\R^{2})\ar[r]^-{\mu_{\R^2}}\ar[d]^{\epsilon_{\C\times \C}} & \A(\R^2)\ar[d]^{\epsilon_{\C}}\\
\A(\C)\times \A(\C) \ar[r]^-{\mu_{\C}} &\A(\C)
}
\]
and it would follow from the second assumption that $F$ would take the monoid $\A(\C)$ to the monoid $\A(\R^2).$

Consider the supermanifold $M = \A(\R^{2}\times \C^{0|2})$
with coordinates $(x, y, \vartheta_1, \vartheta_2)$
and
consider the two maps
$\varphi_{z}, \varphi_{\vartheta_1\vartheta_2}\colon M\rightarrow \A(\C)$
given by $\varphi_z^\sharp (z) = x + iy$
and
$\varphi_{\vartheta_1\vartheta_2}^{\sharp}(z) = \vartheta_1\vartheta_2$,
respectively.
Then we have $\varphi_z = \epsilon_{\C}\circ (x, y)$, so that we would obtain
$F(\varphi_z) = F((x, y)) = (x, y).$

For an arbitrary smooth function $\alpha\colon \R^2\rightarrow \C$ we now define
$f_\alpha\colon M\rightarrow M$ by
\begin{align*}
f_\alpha^\sharp(x) &= x + \alpha \vartheta_1\vartheta_2,\\
f_\alpha^\sharp(y) &= y + (-i)(1-\alpha) \vartheta_1\vartheta_2,\\
f_\alpha^\sharp(\vartheta_i) &= \vartheta_i.
\end{align*}
Then $\varphi_z \circ f_\alpha = \varphi_z + \varphi_{\vartheta_1\vartheta_2}.$
However, on one hand
\begin{align*}
F(\varphi_z \circ f_\alpha)^\sharp &= F(f_\alpha)^{\sharp}\circ F(\varphi_z)^\sharp\\
                                   &= f_\alpha^\sharp\circ F(\varphi_z)^\sharp\\
                                   &= f_\alpha^\sharp\circ (x, y)\\
                                   &= (x, y) + (\alpha\varphi_{\vartheta_1\vartheta_2}, (-i)(1-\alpha)\varphi_{\vartheta_1\vartheta_2})
\end{align*}
and on the other hand
\begin{align*}
F(\varphi_z + \varphi_{\vartheta_1\vartheta_2}) &=  F(\varphi_z) + F(\varphi_{\vartheta_1\vartheta_2})\\
                                                & = (x, y) + F(\varphi_{\vartheta_1\vartheta_2}).
\end{align*}
This would imply
$F(\varphi_{\vartheta_1\vartheta_2}) = (\alpha\varphi_{\vartheta_1\vartheta_2}, (-i)(1-\alpha)\varphi_{\vartheta_1\vartheta_2})$
for arbitrary $\alpha\colon \R^2\rightarrow \C$,
which is absurd.
\end{proof}

Similarly, one proves the following related
proposition.

\begin{Prop}
\label{prop: non-existence of functorial lift from complex functions to two real functions}
The natural transformation $\epsilon_\C\colon \A(\R^2)\rightarrow \A(\C)$ between functors on $\SMan$
admits no section.
\end{Prop}
\begin{proof}
Assume that such a natural transformation $F$ existed.
We use the notation from the previous proof.
We consider again $M = \A(\R^{2}\times \C^{0|2})$
and the two maps
$\varphi_{z}$ and $\varphi_{\vartheta_1\vartheta_2}.$
Then $F(\varphi_z) = (x + n, y + in)$ for a nilpotent function $n$ on $M.$
Defining $f_\alpha$ as previously, we have $\varphi_z \circ f_\alpha = \varphi_z + \varphi_{\vartheta_1\vartheta_2}$,
and so $F(\varphi_z\circ f_\alpha)$ would be independent of $\alpha.$
However, we would have
\begin{align*}
F(\varphi_z\circ f_\alpha) & = f_\alpha^\sharp (x + n, y + in)\\
                           & = (x + \alpha\vartheta_1\vartheta_2 + n, y + (-i)(1-\alpha) \vartheta_1\vartheta_2 + in),
\end{align*}
a contradiction.
\end{proof}

\subsection{Flows of even real vector fields on mixed supermanifolds}
\label{subsec: flows of even real vector fields}
We outline the construction of flows of vector fields on mixed supermanifolds.
In this setting, only even real vector fields can be integrated.
We show that they have a unique maximal flow.

Let $M$ be a mixed supermanifold and let $X$ be an even real vector field.
Let $\mathcal{V}\subseteq \R^1\times M$ be open such that $\{0\}\times M\subseteq \mathcal{V}.$
A morphism
\[
\xymatrix{
\Theta^X\colon \R^{1}\times M\supseteq\mathcal{V}\ar[r] & M
}
\]
is called a \emph{flow of $X$} if
\begin{itemize}
 \item[(a)] $\partial_t \circ {\Theta^X}^\sharp = {\Theta^X}^\sharp \circ X$, and
 \item[(b)] $\Theta^X|_{\{0\}\times M} = \id_M.$
\end{itemize}

Following \cite{GW}, an open subspace $\{0\}\times M\subseteq \mathcal{V}\subseteq \R^1\times M$
such that, for all $m\in M_0$, $\mathcal{V}\cap (\R^1\times \{m\})$ is an interval
and a flow exists on $\mathcal{V}$ is called a \emph{flow domain}.

First we show that a real vector field
on a mixed manifold has a unique maximal flow.
Let $M$ be a mixed manifold and $M^{sm}$ its underlying smooth manifold
which comes with a map $i\colon M^{sm}\rightarrow M.$
Then $(i^*\mathcal{T}_M), \overline{(i^*{\mathcal{T}}_M)}\subseteq \C\otimes \mathcal{T}_{M^{sm}}$
and we have the following exact sequence:
\begin{gather}
\label{equ: exact sequence mixed}
\xymatrix{
 0\ar[r] &i^*(\mathcal{T}_{M, \C}\oplus \bar{\mathcal{T}}_{M, \C})\ar[r] & \C\otimes \mathcal{T}_{M^{sm}}\ar[r] & i^* \mathcal{T}_M/\mathcal{T}_{M, \C}\ar[r] & 0.
}
\end{gather}
In fact, locally
in a neighborhood of the form $(\C^{n_1+n_2}, \R^{n_1}\times \C^{n_2}, \C^{n_2})$,
$i^*\mathcal{T}_{M, \C}$ and $\overline{(i^*\mathcal{T}_{M, \C})}$ are spanned as
$\mathcal{O}_{M^{sm}}$-modules
by $\partial_{z_{i}}$ and $\bar{\partial}_{z_i}$ ($i \in \{n_1 + 1, \ldots, n_1 + n_2\}$),
respectively.

Then we have the following observation.
\begin{Lemma}
For any real vector field $X$ on $M$, there is a unique real vector field $Y$ on $M^{sm}$
such that $(\C\otimes Y)|_{\mathcal{O}_M} = X.$
\end{Lemma}
\begin{proof}
Consider two such real vector fields $Y_1$ and $Y_2$ on $M^{sm}.$
Locally on the model space defined by $(\C^{n_1 + n_2}, \R^{n_1}\times \C^{n_2}, \C^{n_2})$,
with coordinates $\{x = (x_1, \ldots x_{n_1}), z = (z_{1}, \ldots, z_{n_2})\}$,
we have
\[
 X = \sum_i f_i(x) \partial_{x_i} + \sum_j g_j(x, z) \partial_{z_j}
\]
for smooth real functions $f_i(x)$ and partially holomorphic functions $g_j(x, z).$
Hence we have
\[
 Y_l = \sum_i f_i(x) \partial_{x_i} + \sum_j g_j(x, z) \partial_{z_j} + \sum_j \bar{g}_j(x, z) \bar{\partial}_{z_j}\quad (l\in \{1, 2\}),
\]
which proves uniqueness.
In order to prove existence, we choose a splitting of (\ref{equ: exact sequence mixed})
in order to write $i^*X = X_\R + X_\C$, where $X_\C\in i^*\mathcal{T}_{M, \C}.$
Then $Y = X_\R + X_\C + \bar{X}_\C$ is the desired vector field.
\end{proof}

\begin{Lemma}
Let $(V, V_\R, V_\C)$ be a mixed vector space and let
$X$ be a real vector field on $U\subseteq \A(V_\R)$
and $Y$ the unique real vector field such that $(\C\otimes Y)|_{\mathcal{O}_U} = X.$
The maximal flow $\Theta^Y\colon \mathcal{V}^{sm}\rightarrow U^{sm}$ of $Y$
defines a morphism of mixed manifolds $\Theta^X$
which is the unique maximal flow of $X.$
\end{Lemma}
\begin{proof}
The proof of \cite[Theorem 12.4.2]{BER} applies to show that
for every $p\in U$ there is an open neighbourhood
$U'$ of $p$ and an $\epsilon > 0 $ such that $(-\epsilon, \epsilon)\times U'\subseteq \mathcal{V}^{sm}$
and
$\Theta^Y|_{(-\epsilon, \epsilon)\times U'}$ is a mixed morphism.
Since $\mathcal{V}^{sm}$ is a flow domain
and the since flow is additive
we conclude that $\Theta^Y$ defines a mixed morphism.
This is a flow morphism since $\Theta^Y$ is a flow for $Y$
and $(\C\otimes Y)|_{\mathcal{O}_U} = X.$
Uniqueness follows from uniqueness of the flow of $Y$ and
maximality is ensured by maximality of $\mathcal{V}^{sm}.$
\end{proof}

\begin{Lemma}
Let $(V, V_\R, V_\C)$ be a mixed super vector space
and let $X$ be a real even vector field on the open subspace
$U\subseteq \A(V_\R).$
Furthermore, let $\tilde{X}$ be the underlying real vector field on $\A((V_{\bar{0}})_\R)$
with maximal flow $\Theta^{\tilde{X}}\colon \mathcal{V}_0\rightarrow U_0.$
There is a unique flow morphism $\Theta^X\colon \mathcal{V}\rightarrow U$
where $\mathcal{V}_{0}$ is the maximal flow domain and $(\Theta^{X})_{0}$ is
the maximal flow of $\tilde{X}.$
\end{Lemma}
\begin{proof}
Following the proof given in \cite[Lemma 2.1]{GW},
the higher order terms of the flow $\Theta^{X}$ are constructed
by solving linear ordinary differential equations.
The unique solutions will automatically be partially holomorphic,
since the initial condition, the identity, is partially holomorphic.
So we get a flow $\Theta^X\colon \mathcal{V}\rightarrow U$ for $X$
with $(\Theta^X)_0 = \Theta^{\tilde{X}}$ and $\mathcal{V}\subseteq \R\times U$
is the open sub supermanifold with base $\mathcal{V}_0.$ 
\end{proof}

By the same reasoning as in \cite[Lemma 2.2]{GW}
one can prove the existence of flow domains:
\begin{Lemma}
Let $X$ be an even real vector field on the mixed supermanifold $M.$
Then there exists a flow domain $\mathcal{V}$ for $X.$
Furthermore, if $\mathcal{V}_i$, $i\in \{1, 2\}$, are flow
domains with flows $\Theta^X_{i}$, then 
$\Theta^X_1 |_{\mathcal{V}_1\cap \mathcal{V}_2} = \Theta^X_2 |_{\mathcal{V}_1\cap \mathcal{V}_2}.$
\end{Lemma}

Putting everything together we obtain the final result.

\begin{Thm}
\label{thm: maximal flow}
Let $X$ be an even real vector field on the mixed supermanifold $M$
with underlying real vector field $\tilde{X}$ on $M_0.$
Then there exists a unique flow map $\Theta^X\colon \mathcal{V}\rightarrow M$
where $\mathcal{V}$ is the maximal flow domain for $X.$
Moreover, $(\Theta^X)_0$ is the maximal flow of $\tilde{X}.$
\end{Thm}
\begin{proof}
This follows from the above considerations by taking the union of all flow domains.
\end{proof}

\begin{Def}
 An even real vector field is called complete if its maximal flow domain
$\mathcal{V}$ equals $\R\times M.$
\end{Def}

The following basic properties can be proved as in the classical case.
\begin{Prop}
\label{prop: commuting flow and vector fields}
Suppose $X$ is an even real vector field and $Y$ is an arbitrary vector field on $M.$
 \begin{itemize}  \item[(a)] $\mathfrak{L}_X Y := {\partial_t}|_{t = 0} (\Theta^X_{t})^\sharp \circ Y \circ (\Theta^X_{-t})^\sharp = [X, Y].$
  \item[(b)] If $[X, Y] = 0$, then ${\Theta^X}^{\sharp}$ and $Y$ commute.
 \end{itemize}
\end{Prop}
\begin{proof}
See for instance \cite[Corollary 3.7]{Bergner}.
\end{proof}


\newpage
\begin{thebibliography}{22}
\footnotesize
    \bibitem{Cortes et al}
    D.V.~Alekseevsky, V.~Cort\'es, C.~Devchand, and U.~Semmelmann \emph{Killing spinors are Killing vector fields in Riemannian Supergeometry.}
    J. Geom. Phys. 26, No.1-2, 37-50 (1998).

    \bibitem{AHW}
     A.~Alldridge, J.~ Hilgert, and T.~Wurzbacher \emph{Singular superspaces.}
     Math. Z. 278, No. 1-2, 441-492 (2014). 
    
    \bibitem{Ballmann}
    W.~Ballmann \emph{Automorphism groups of manifolds.}
    Lecture notes, Bonn (2011).

   \bibitem{Eastwood}
    T.~N.~Bailey, M.~G.~Eastwood, and S.~G.~Gindikin \emph{Nonclassical descriptions of analytic cohomology.}
    Bure\v{s}, Jarol\'{\i}m (ed.), The proceedings of the 22nd winter school ``Geometry and physics", Srn\'{\i}, Czech Republic, January 12-19, 2002. Palermo: Circolo Matemàtico di Palermo, Suppl. Rend. Circ. Mat. Palermo, II. Ser. 71, 67-72 (2003). 
    
    \bibitem{BER}
    M.~Salah Baouendi, P.~Ebenfelt, and L.~Preiss Rothschild \emph{Real Submanifolds in Complex Space and Their Mappings.}    
    Princeton Mathematical Series. Princeton, NJ: Princeton University Press. xii, 404 p. (1999).

    \bibitem{Bergner}
    H.~Bergner \emph{Globalizations of infinitesimal actions on supermanifolds.}
    Journal of Lie Theory 24, No.3, 809-847 (2013).

    \bibitem{CCF}
    C.~Carmeli, L.~Caston, and R.~Fioresi \emph{Mathematical Foundations of Supersymmetry.}
    EMS Series of Lectures in Mathematics. Z\"urich: European Mathematical Society (EMS). xiii, 287~p. (2011).
    
    \bibitem{Deligne Morgan}
    P.~Deligne, J.~W.~Morgan \emph{Notes on supersymmetry (following Joseph Bernstein). Quantum fields and strings: a course for mathematicians},
    Vol. 1, 2 (Princeton, NJ, 1996/1997), 41?97, Amer. Math. Soc., Providence, RI (1999).

    \bibitem{Fujio}
    M.~Fujio \emph{Super G-structures of finite type.}
    Osaka J. Math. 28, No.1, 163-211 (1991).

    \bibitem{GW}
    S.~Garnier, and T.~Wurzbacher \emph{The geodesic flow on a Riemannian supermanifold.}
    J. Geom. Phys. 62, No. 6, 1489-1508 (2012).

    \bibitem{Goertsches}
    O.~Goertsches \emph{Riemannian Supergeometry.}
    Math. Z. 260, No. 3, 557-593 (2008).

    \bibitem{Kobayashi}
    S.~Kobayashi \emph{Transformation groups in differential geometry. Reprint of the 1972 ed.}
    Classics in Mathematics. Berlin: Springer-Verlag. viii, 182 p. (1995).
    
    \bibitem{Lott}
    J.~Lott \emph{Torsion constraints in supergeometry.}
    Commun. Math. Phys. 133, No.3, 563-615 (1990).

     \bibitem{Zirnbauer}
     M.R.~Zirnbauer \emph{Riemannian symmetric superspaces and their origin in random-matrix theory.}
     J. Math. Phys. 37, No.10, 4986-5018 (1996).
         
     \bibitem{Witten}
     E.~Witten \emph{Notes On Supermanifolds and Integration.}
     arxiv:1209.2199 (2012).
\end{thebibliography}
\end{document}